\documentclass[francais]{smfart}

\usepackage[english,frenchb]{babel}
\usepackage[latin1]{inputenc}
 \usepackage[T1]{fontenc}
 \usepackage{lmodern}

\usepackage{amsmath}
\usepackage{amsfonts}
\usepackage{amssymb}
\usepackage{amsthm}

% ...

% ...

\newtheoremstyle{monThm}
{0cm} 
{0cm} 
{\itshape} 
{} 
{\scshape} 
{.} 
{0cm} 
{} 

\theoremstyle{monThm}
\newtheorem{Thm}{Th\'eor\`eme}

\newtheorem{CoroFonda}[Thm]{Corollaire}

\theoremstyle{theorem}
\newtheorem*{CiteThm}{Th\'eor\`eme}
\newtheorem*{CiteCoro}{Corollaire}
\newtheorem{proposition}{Proposition}[section]
\newtheorem{lemme}[proposition]{Lemme}
\newtheorem{corollaire}[proposition]{Corollaire}
\newtheorem*{conjecture}{Conjecture}

\theoremstyle{definition}
\newtheorem{definition}[proposition]{Définition}
\theoremstyle{remark}
\newtheorem*{rmq}{Remarque}

\newcommand{\R}{\mathbb{R}}
\newcommand{\C}{\mathbb{C}}

\renewcommand{\H}{\mathbb{H}}

\newcommand{\p}{\mathfrak{p}}
\renewcommand{\sl}{\mathfrak{sl}}

\newcommand{\Vect}{\mathrm{Vect}}

\newcommand{\Isom}{\mathrm{Isom}}
\newcommand{\Aut}{\mathrm{Aut}}

\newcommand{\Id}{\mathrm{Id}}

\newcommand{\Tr}{\mathrm{Tr}}
\newcommand{\Ad}{\mathrm{Ad}}
\newcommand{\ad}{\mathrm{ad}}

\newcommand{\norm}[1]{\left \Vert #1 \right \Vert}

\renewcommand{\d}{\mathrm{d}}

\newcommand{\gammapoint}{\stackrel{\cdot}{\gamma} \!}

\newcommand{\dev}{\mathrm{dev}}

\renewcommand{\epsilon}{\varepsilon}

\renewcommand{\l}{\mathfrak{l}}
\renewcommand{\k}{\mathfrak{k}}
\newcommand{\m}{\mathfrak{m}}
\newcommand{\n}{\mathfrak{n}}
\newcommand{\com}{\mathfrak{com}}
\newcommand{\1}{\mathbf{1}}
\newcommand{\Aff}{\textrm{Aff}}
\newcommand{\Diff}{\textrm{Diff}}
\newcommand{\Sym}{\textrm{Sym}}
\newcommand{\T}{\mathrm{T}}

\newcommand{\PSL}{\mathrm{PSL}}
\newcommand{\PGL}{\mathrm{PGL}}
\renewcommand{\O}{\mathrm{O}}
\newcommand{\U}{\mathrm{U}}
\newcommand{\PSO}{\mathrm{PSO}}
\newcommand{\SL}{\mathrm{SL}}
\newcommand{\SO}{\mathrm{SO}}
\newcommand{\SU}{\mathrm{SU}}
\newcommand{\CP}{\mathbb{C}\mathrm{P}}
\renewcommand{\phi}{\varphi}

\theoremstyle{definition}
\newtheorem*{definition*}{Définition}

\title[Complétude des variétés localement symétriques]{Sur la complétude de certaines variétés pseudo-riemanniennes localement symétriques}

\alttitle{Completeness of some semi-riemannian locally symetric manifolds}

\author{Nicolas Tholozan}
\address{Universit\'e de Nice-Sophia Antipolis, Laboratoire J.-A. Dieudonn\'e, UMR 7351 CNRS, Parc Valrose, 06108 Nice Cedex 2, France} 
\email{tholozan@unice.fr}

\begin{document}

\begin{abstract} 
Nous prouvons que certains espaces pseudo-riemanniens symétriques n'admettent pas d'ouvert strict divisible par l'action d'un groupe discret d'isométries. Autrement dit, si une variété pseudo-riemannienne compacte est localement isométrique à un tel espace, et si son application développante est injective, alors la variété est géodésiquement complète, et donc isométrique à un quotient de l'espace modèle tout entier. Ces résultats étendent, sous une hypothèse supplémentaire (l'injectivité de l'application développante), les théorèmes de Carrière et Klingler selon lesquels les variétés lorentziennes compactes de courbure constante sont géodésiquement complètes.
\end{abstract}

\begin{altabstract}
We prove that some symetric semi-riemannian manifolds do not admit a proper domain which is divisible by the action of a discrete group of isometries. In other words, if a closed semi-riemannian manifold is locally isometric to such a model, and if its developing map is injective, then the manifold is actually geodesically complete, and therefore isometric to a quotient of the whole model space. Those results extend, under additional hypothesis (the injectivity of the developing map), the theorems of Carrière and Klingler stating that closed lorentz manifolds of constant curvature are geodesically complete.
\end{altabstract}

\maketitle

\subsection*{Introduction}

Soit $X$ une variété lisse munie d'une action transitive d'un groupe de Lie $G$ de dimension finie. Une variété $M$ est dite \emph{localement modelée} sur $X$ lorsqu'elle est munie d'un atlas de cartes à valeurs dans $X$ dont les changements de cartes sont des restrictions de transformations de $G$. On dit aussi que $M$ est munie d'une $(G,X)$-structure, ou encore que $M$ est une $(G,X)$-variété, selon la terminologie de Thurston (voir section \ref{Rappel}). Certaines $(G,X)$-variétés apparaissent naturellement lorsque $M$ possède une structure géométrique rigide localement homogène. Par exemple, une variété munie d'une métrique riemannienne plate est localement modelée sur l'espace euclidien. Il existe de même une notion de métrique lorentzienne de courbure nulle (voir par exemple \cite{Wolf74}, p.63), et les variétés munies de telles métriques sont localement modelées sur l'espace de Minkovski $\R^{n-1,1}$, c'est-à-dire $\R^n$ munie d'une métrique de signature $(n-1,1)$ invariante par translation.

Certaines $(G,X)$-structures s'obtiennent en quotientant un ouvert $U$ du modèle $X$ par un sous-groupe discret de $G$ agissant librement et proprement discontinûment sur $U$. De telles structures sont dites \emph{kleiniennes} (voir section \ref{Kleinienne}). Dans le cas particulier où le domaine $U$ est le modèle $X$ tout entier, la $(G,X)$-structure est dite \emph{complète}. Lorsque $G$ préserve une métrique riemannienne sur $X$, il découle du théorème de Hopf-Rinow que toutes les $(G,X)$-variétés compactes sont complètes (section \ref{Geodesique}). Mais ce résultat n'est plus vrai pour d'autres espaces homogènes non riemanniens. Ainsi, en conséquence du théorème d'uniformisation de Poincaré--Koebe, toute surface de Riemann peut être munie d'une structure projective complexe (i.e. une $(\PSL(2,\C), \CP^1)$-structure) kleinienne compatible avec sa structure complexe. Toutefois, seule la sphère de Riemann possède une structure complète.

La question de savoir sous quelles hypothèses les $(G,X)$-variétés compactes sont complètes est une question ouverte. En particulier, dès que $G$ préserve une forme volume sur $X$, on ne connaît aucun exemple de $(G,X)$-variété compacte non complète. La conjecture la plus célèbre dans ce domaine est la conjecture de Markus, selon laquelle toute variété affine compacte possédant une forme volume parallèle est complète (voir section \ref{Markus}).

La conjecture de Markus a été prouvée par Carrière dans le cas particulier où la variété affine possède une forme quadratique de signature $(n-1,1)$ parallèle; autrement dit, pour les variétés lorentziennes plates (\cite{Carriere89}). Nous l'étendons ici aux variétés lorentz-hermitiennes plates, avec toutefois une hypothèse supplémentaire.\\

\begin{Thm} \label{ThmAffine}
Toutes les $(\U(n-1,1) \ltimes \C^n, \C^n)$-variétés kleiniennes compactes sont complètes.\\
\end{Thm}

(Le groupe $\U(n-1,1)\ltimes \C^n$ désigne le groupe des transformations affines complexes de $\C^n$ dont la partie linéaire préserve une forme sesquilinéaire de signature réelle $(2n-2,2)$.)\\

Le théorème de Carrière a été généralisé dans une autre direction par Klingler dans \cite{Klingler96}. Ce dernier prouve que les variétés lorentziennes compactes de courbure constante sont complètes. Pour ces variétés, la complétude de la $(G,X)$-structure est équivalente à la complétude du flot géodésique de la métrique (voir section \ref{Geodesique}). On peut plus généralement conjecturer que toutes les variétés pseudo-riemanniennes compactes de courbure constante, et même que toutes les variétés pseudo-riemanniennes compactes localement symétriques (au sens de Cartan, voir section \ref{Geodesique} ou \cite{Wolf74}, p.57) sont géodésiquement complètes.

Dans \cite{DumitrescuZeghib09}, Dumitrescu et Zeghib s'intéressent aux variétés complexes compactes de dimension 3 munies de métriques riemanniennes holomorphes (i.e. des sections holomorphes partout non dégénérées du fibré des formes quadratiques complexes sur le fibré tangent). Ils prouvent notamment que de telles variétés possèdent toujours une métrique riemannienne holomorphe de courbure constante, ramenant la classification de ces variétés à l'étude de deux types de structures localement homogènes. Les premières (celles de courbure nulle)  sont localement modelées sur $\C^3$, muni de l'action affine complexe du groupe $\SO(3,\C) \ltimes \C^3$. Quant aux secondes (celles de courbure constante non nulle), elles sont localement modelées sur $\SL(2,\C)$ muni de l'action de $\SL(2,\C) \times \SL(2,\C)$ par translations à gauche et à droite. Il n'existe malheureusement pas d'analogue holomorphe des théorèmes de Carrière et Klingler, ce qui constitue le principal obstacle à une classification topologique des 3-variétés complexes compactes possédant des métriques riemanniennes holomorphes. Nous prouvons toutefois ici deux résultats qui étendent partiellement les théorèmes de Carrière et Klingler au contexte holomorphe.\\

\begin{Thm} \label{ThmHoloPlat}
Toutes les $(\SO(3,\C) \ltimes \C^3, \C^3)$-variétés kleiniennes compactes sont complètes.\\
\end{Thm}

(Le groupe $\SO(3,\C) \ltimes \C^3$ désigne le groupe des transformations affines complexes de $\C^3$ dont la partie linéaire préserve une forme $\C$-bilinéaire symétrique non dégénérée.)\\

\begin{Thm} \label{ThmGrpLie}
Soit $L$ un groupe de Lie semi-simple de rang $1$. Considérons $L \times L$ agissant sur $L$ par translations à gauche et à droite. Alors toutes les $(L\times L, L)$-structures kleiniennes compactes sont complètes.\\
\end{Thm}
Le théorème \ref{ThmGrpLie} s'applique en particulier à $\SL(2,\C)$. Dans sa version plus générale, il fait pendant à des résultats récents de Kassel (\cite{Kassel08}) et Guéritaud--Guichard--Kassel--Wienhard (\cite{GueritaudKassel}, \cite{GGKW}) qui décrivent les quotients compacts d'un groupe de Lie $L$ par un sous-groupe discret de $L\times L$ (voir section \ref{Kassel}). En particulier, les auteurs prouvent dans \cite{GGKW} que, dans l'espace des $(L\times L, L)$-structures sur une variété $M$, le domaine des structures complètes forme un ouvert\footnote{La topologie de l'espace des $(G,X)$-structures sur une variété est présentée sommairement à la section \ref{GXDeformation}. On consultera par exemple \cite{Goldman10} pour plus de détails.}. Or, notre théorème a pour conséquence que ce domaine est aussi fermé (corollaire \ref{Fermeture}). On en déduit le corollaire suivant:\\

\begin{CoroFonda} \label{CoroConnexe}
Soit $L$ un groupe de Lie semi-simple de rang $1$, et $M$ une variété compacte de même dimension que $L$. Alors, dans l'espace des $(L\times L, L)$-structures sur $M$, le domaine des structures complètes forme une union de composantes connexes.\\
\end{CoroFonda}
Autrement dit, il est impossible de déformer continûment une structure complète en une structure incomplète. C'est cette question qui a initialement motivé nos travaux. Toutefois, nous ne savons pas s'il existe des variétés compactes possédant des composantes connexes de structures incomplètes. Il pourrait par exemple exister des variétés compactes possédant des $(L\times L, L)$-structures, mais ne possédant aucune structure complète.\\

La preuve des théorèmes \ref{ThmAffine} et \ref{ThmHoloPlat} s'inspire des idées du théorème de Carrière. En supposant qu'il existe un ouvert strict $U$ de $\C^n$ muni d'une action proprement discontinue d'un sous-groupe $\Gamma$ de $\U(n-1,1)\ltimes \C^n$ (ou $\SO(3,\C) \ltimes \C^3)$, on remarque que le bord de $U$ doit contenir des ``ellipsoïdes dégénérés'' sous l'action de $\Gamma$. La principale différence avec le théorème de Carrière est que cette action est de discompacité $2$: les ellisoïdes dégénèrent non pas sur des hyperplans réels mais sur des hyperplans complexes, de codimension réelle $2$. La preuve du théorème \ref{ThmGrpLie} repose sur un raisonnement analogue, mais sa transposition à la dynamique de l'action de $L\times L$ sur $L$ nécessite des résultats classiques sur la structure des groupes de Lie de rang $1$.

Nos théorèmes ne s'affranchissent pas de l'hypothèse que la $(G,X)$-structure considérée est a priori kleinienne, et la question plus générale de la complétude des $(G,X)$-structures compactes reste ouverte. Nous espérons toutefois esquisser à travers nos preuves une approche systématique du problème.\\

L'article s'organise comme suit. Dans la section \ref{Rappel}, nous rappelons la définition de la notion de $(G,X)$-structure, nous décrivons l'espace de déformation des $(G,X)$-structures sur une variété compacte à partir du principe d'Ehresman-Thurston, puis nous précisons le rapport entre la notion de complétude et celle de complétude géodésique pour des variétés pseudo-riemanniennes; enfin, nous donnons la définition d'une structure kleinienne et son rapport avec la notion de domaine divisible.

Dans la section \ref{Affine}, nous nous focalisons sur la géométrie affine. Nous rappelons la conjecture de Markus et les différents résultats qui la concernent, et nous prouvons les théorèmes \ref{ThmAffine} et \ref{ThmHoloPlat}.

Dans la section \ref{GroupeLie}, nous présentons la géométrie des groupes de Lie semi-simples de rang $1$, qui, munis de leur métrique de Killing, sont des espaces pseudo-riemanniens symétriques. Nous recensons les divers résultats qui décrivent les quotients compacts de $L$ par des sous-groupes de $L\times L$. Nous rappelons ensuite quelques propriétés des groupes de Lie de rang $1$, qui nous permettent d'adapter la preuve du théorème \ref{ThmAffine} et de prouver le théorème \ref{ThmGrpLie}.

\subsubsection*{Remerciements}
L'auteur tient à remercier son directeur de thèse, Sorin Dumitrescu, ainsi que Fanny Kassel, Yves Benoist et Charles Frances pour les discussions fructueuses sans lesquelles ce travail n'aurait pas vu le jour.

\subsection*{Notions préliminaires: actions propres, libres et cocompactes}

Nous rappelons ici quelques définitions classiques concernant les actions de groupes, et les adaptons au cadre qui nous intéresse.\\

Soit $X$ une variété différentielle munie d'une action lisse et transitive d'un groupe de Lie $G$ de dimension finie. Quitte à quotienter $G$ par le noyau du morphisme de $G$ dans $\Diff (X)$ donné par l'action, on pourra supposer que cette action est \emph{fidèle}, c'est-à-dire qu'aucun élément non trivial de $G$ n'agit trivialement sur $X$. 

Considérons maintenant un ouvert $U$ de $X$ et un sous-groupe $\Gamma$ de $G$ dont l'action sur $X$ préserve $U$. L'action de $\Gamma$ sur $U$ est dite \emph{propre} si pour tout compact $K$ de $U$, l'ensemble des $g \in \Gamma$ tels que $g\cdot K$ intersecte $K$ est relativement compact dans $G$.

\begin{proposition}
L'action de $\Gamma$ sur $U$ est propre si et seulement si, pour toute suite $x_n$ dans $U$ convergeant vers $x\in U$ et toute suite $\gamma_n$ dans $\Gamma$, on a l'implication suivante:
$$\gamma_n \cdot x_n \textrm{ bornée dans $U$} \Rightarrow \gamma_n \textrm{ bornée dans $G$}.$$
\end{proposition}

En particulier, si l'action de $\Gamma$ préserve une distance sur $U$ adaptée à la topologie de variété de $U$, alors cette action est propre. Il est clair que lorsque $\Gamma$ est discret, son action sur $U$ est propre si et seulement si elle est \emph{proprement discontinue}, c'est-à-dire que pour tout compact $K$ de $U$, l'ensemble des $g \in \Gamma$ tels que $g\cdot K$ intersecte $K$ est fini.

L'action de $\Gamma$ sur $U$ est dite \emph{libre} si le stabilisateur de tout point est réduit à l'identité. L'action de $\Gamma$ sur $U$ est proprement discontinue et libre si et seulement si le quotient $\Gamma \backslash U$ est une variété de même dimension que $U$.

Pour qu'une action proprement discontinue soit libre, il suffit qu'elle soit sans torsion. Rappelons que, d'après le théorème de Selberg (\cite{Selberg89}), un sous-groupe de type fini d'un groupe linéaire est virtuellement sans torsion. Dans les cas qui nous concernent, on perdra donc peu de généralité en supposant que les actions propres et discrètes sont également libres.

Pour finir, on dira que l'action de $\Gamma$ sur $U$ est \emph{cocompacte} s'il existe un compact $K$ de $U$ tel que $\Gamma \cdot K$ recouvre $U$. Autrement dit, $\Gamma$ agit cocompactement sur $U$ si le quotient de $U$ par $\Gamma$ est \emph{paracompact} (i.e. ``compact sans être séparé'').

Lorsque l'action de $\Gamma$ sur $U$ est propre discontinue, le quotient est nécessairement séparé, et l'action est cocompacte si et seulement si le quotient est compact. Toutefois, la propriété de cocompacité telle que nous l'avons énoncée fournit de nombreuses informations même lorsque le quotient n'est pas séparé.

\section{Les $(G,X)$-structures et leur espace de déformation} \label{Rappel}

\label{GXstructure}

Considérons $G$ un groupe de Lie (de dimension finie), et $H$ un sous-groupe fermé de $G$, de sorte que le quotient $X = G/H$ soit un espace $G$-homogène (i.e. une variété lisse munie d'une action transitive de $G$). Soit enfin $M$ une variété topologique de même dimension que $X$. Thurston introduit dans \cite{Thurston80} la notion de $(G,X)$-structure, dont l'idée remonte à Ehresmann.

\begin{definition}
Une $(G,X)$-structure sur $M$ est la donnée d'un atlas $(U_i, \phi_i)_{i\in I}$ où $(U_i)_{i\in I}$ est un recouvrement ouvert de $M$ et $\phi_i$ un difféomorphisme de $U_i$ dans un ouvert de $X$, tels que pour tous $i,j$, sur chaque composante connexe de $U_i \cap U_j$, il existe $g \in G$ tel que 
$${\phi_j} = g\cdot  {\phi_i}.$$
\end{definition}
On dira aussi que $M$ est localement modelée sur $X$, ou encore que $M$ est une $(G,X)$-variété.

\begin{rmq}
Si $X$ n'est pas simplement connexe, il existe une extension $\tilde{G}$ de $G$ qui agit transitivement sur $\tilde{X}$, et toute $(G,X)$-structure induit une $(\tilde{G}, \tilde{X})$-structure. On peut donc sans perte de généralité supposer que $X$ est simplement connexe, ce que nous ferons, sauf mention du contraire.
\end{rmq}

Une $(G,X)$-structure sur $M$ induit naturellement une structure de variété lisse. Plus généralement, si $M$ est localement modelée sur $X$, toute structure géométrique\footnote{La notion de \emph{structure géométrique} a été formalisée par Gromov dans \cite{Gromov88}. On peut y penser ici dans un sens vague comme un objet attaché à la variété, tel qu'une forme différentielle ou une métrique pseudo-riemannienne.} sur $X$ invariante sous l'action de $G$ induit une structure géométrique sur $M$ qui lui est localement isomorphe. Par exemple, si $G$ préserve une métrique pseudo-riemannienne sur $X$, une $(G,X)$-structure sur $M$ induit une métrique pseudo-riemannienne sur $M$, telle que les cartes locales de l'atlas de la $(G,X)$-structure sont des isométries locales.

Réciproquement, soit $X$ une variété munie d'une métrique pseudo-riemannienne homogène. Supposons de plus que toutes les isométries locales de $X$ dans $X$ se prolongent en des isométries globales. Soit $M$ une variété pseudo-riemannienne localement isométrique à $X$. Alors un atlas $(U_i, \phi_i)$ formé d'isométries locales de $M$ dans $X$ induit naturellement une $(\Isom(X),X)$-structure sur $M$.

Par exemple, une $(SL(2,\R), \H^2)$-variété est une surface munie d'une métrique de courbure $-1$. Les $(G,X)$-structures que nous étudierons dans la suite sont des $(G,X)$-structures pseudo-riemanniennes.

\subsection{Espace des déformations des $(G,X)$-structures sur une variété}
\label{GXDeformation}
\begin{proposition}
Une $(G,X)$-structure sur $M$ induit une paire $(\dev,\rho)$ où $\rho$ est un morphisme de $\pi_1(M) \to G$ et $\dev$ un difféomorphisme local de $\tilde{M}$ dans $X$ qui est $\rho$-équivariant, c'est-à-dire tel que le diagramme suivant commute:

\begin{displaymath}
\begin{array}{rcl}
\tilde{M} & \stackrel{\dev}{\longrightarrow} & X \\
g\downarrow & \ & \downarrow \rho(g) \\
\tilde{M} & \stackrel{\dev}{\longrightarrow} & X \\
\end{array}
\end{displaymath}

$\dev$ est appelée une application développante et $\rho$ un morphisme d'holonomie pour la $(G,X)$-variété $M$. Réciproquement, la donnée d'un couple $(\dev,\rho)$ induit une $(G,X)$-structure sur $M$.
\end{proposition}

Cependant, à une $(G,X)$-structure ne correspond pas un unique couple $(\dev,\rho)$. Le groupe $G$ agit sur l'ensemble des couples $(\dev,\rho)$ par
$$g \cdot (\dev, \rho) = (g\circ \dev, g \rho g^{-1}),$$
et les orbites de $G$ sont exactement les classes de couples $(\dev, \rho)$ correspondant à une même $(G,X)$-structure.

En outre, on veut identifier deux $(G,X)$-structures sur une variété $M$ lorsque l'une est l'image de l'autre par un difféomorphisme de $M$ homotope à l'identité. L'\emph{espace de déformation} des $(G,X)$-structures sur $M$ sera donc l'ensemble :

$$\mathrm{Def}_{(G,X)}(M) = G \backslash \{(\dev,\rho), \dev : \tilde{M} \to X\ \rho\textrm{-équivariant} \} / \Diff_0(M).$$
\\

Cet espace peut être muni de la  topologie induite par passage au quotient de la topologie de la convergence $C^1$ sur tout compact de l'application développante, ce qui en fait un espace topologique séparé. Le théorème suivant affirme que, pour une variété compacte $M$, l'application qui à une $(G,X)$-structure sur $M$ associe son morphisme d'holonomie est un homéomorphisme local de $\mathrm{Def}_{(G,X)}(M)$ dans l'espace des représentations de $\pi_1(M) \to G$ modulo conjugaison.

\begin{CiteThm}[Ehresmann, Thurston,\cite{Thurston80}]

Soit $M$ une variété compacte. L'espace $\mathrm{Def}_{(G,X)}(M)$ est localement paramétré par les morphismes d'holonomie, dans le sens suivant.
Soit $(\dev,\rho)$ une $(G,X)$-structure sur $M$. Si $\rho'$ est un morphisme de $\pi_1(M)$ dans $G$ suffisamment proche de $\rho$, alors $\rho'$ est l'holonomie d'une $(G,X)$-structure proche de $(\dev, \rho)$. De plus, s'il existe $\dev'_1$ et $\dev'_2$ proches de $\dev$ et telles que $(\dev'_1, \rho')$ et $(\dev'_2, \rho')$ sont des $(G,X)$-structures, alors il existe un élément $\phi$ de $\Diff_0(M)$ tel que $\dev'_2 = \dev'_1 \circ \tilde{\phi}$.
\end{CiteThm}

De plus, comme $M$ est compacte, son groupe fondamental est de présentation finie, et $\mathrm{Def}_{(G,X)}(M)$ est donc localement homéomorphe à une variété algébrique de dimension finie.

\subsection{Complétude des $(G,X)$-structures}
\label{Geodesique}

Une $(G,X)$-structure sur $M$ est dite \emph{complète} si son application développante est un difféomorphisme global de $\tilde{M}$ dans $X$. Lorsqu'une $(G,X)$-structure est complète, l'application développante identifie le revêtement universel de $M$ au modèle $X$, sur lequel $\pi_1(M)$ agit discrètement, proprement et librement via l'holonomie. Les $(G,X)$-variétés complètes sont donc les quotients de $X$ par un sous-groupe discret de $G$ agissant proprement et librement.

Le terme de \emph{complétude} provient du rapport, dans le cas de $(G,X)$-structures pseudo-riemanniennes, avec la complétude géodésique. Rappelons qu'à l'instar des métriques riemanniennes, les métriques pseudo-riemanniennes possèdent une unique \emph{connexion de Levi-Civita} $\nabla$. Les courbes solutions de l'équation $\nabla_{\gammapoint} \gammapoint \ = 0$ sont appelées géodésiques de la métrique (même si elles ne se comprennent plus comme des courbes minimisantes).
Le flot géodésique peut être vu comme le flot d'un champ de vecteur sur l'espace total du fibré tangent à la variété, et la métrique est dite \emph{géodésiquement complète} si le flot de ce champ de vecteur est complet, c'est-à-dire s'il est défini sur tout $TM \times \R$.

\begin{proposition} \label{PropGeodComp}
Supposons que $G$ agisse transitivement sur une variété $X$ en préservant une métrique pseudo-riemannienne $g_X$ géodésiquement complète. Toute $(G,X)$-variété $M$ hérite naturellement d'une métrique pseudo-riemannienne $g_M$ localement isométrique à $g_X$. Alors la $(G,X)$-structure de $M$ est complète si et seulement si $g_M$ est géodésiquement complète.
\end{proposition}

\begin{proof}
Par définition de $g_M$, l'application développante est une isométrie locale, et envoie donc une géodésique sur une géodésique. Supposons que la $(G,X)$-structure sur $M$ est complète. Alors $\dev : (\tilde{M}, g_{\tilde{M}}) \to (X, g_X)$ est une isométrie globale, et $g_{\tilde{M}}$ (et donc $g_M$) est géodésiquement complète si et seulement si $g_X$ l'est aussi.

Réciproquement, si $g_X$ et $g_M$ sont géodésiquement complètes, montrons que l'application développante est un revêtement. Comme c'est un difféomorphisme local, il suffit de vérifier qu'elle possède la propriété de relèvement des chemins. Par l'absurde, supposons qu'il existe un chemin continu $c: [0,1] \to X$ et un point $m \in \dev^{-1}(c(0))$ tel que $c$ se relève en partant de $m$ sur $[0,1)$ mais pas sur $[0,1]$. Soit $\tilde{c} : [0,1) \to \tilde{M}$ ce relèvement. 

Comme le chemin $c$ dans $X$ est compact, il existe un temps $t$ pour lequel est $c([t,1])$ est inclus dans le domaine d'injectivité du flot géodésique partant de $c(t)$. Autrement dit, il existe un ouvert $U \in T_{c(t)}X$ contenant $0$ tel que le flot géodésique ${\exp_x}_{|U}$ est un difféomorphisme local injectif dont l'image contient $c([t,1])$. Soit alors $\tilde{U}$ l'ouvert de $T_{\tilde{c}(t)}\tilde{M}$ image réciproque de $U$ par la différentielle de l'application développante. Comme l'application développante envoie géodésique sur géodésique, le diagramme suivant commute : 

\begin{displaymath}
\begin{array}{rcl}
\tilde{U} & \stackrel{\d_{\tilde{c}(t)}\dev}{\longrightarrow} & U \\
\exp_{\tilde{c}(t)}\downarrow & \ & \downarrow \exp_{c(t)} \\
\tilde{M} & \stackrel{\dev}{\longrightarrow} & X \\
\end{array}
\end{displaymath}

Posons alors 
\begin{eqnarray}
\hat{c}(u) & = & \tilde{c}(u)\quad \textrm{si $u\leq t$} \\
\ & = & \exp_{\tilde{c}(t)} \circ \d_{c(t)} \dev^{-1} \circ \exp_{c(t)}^{-1}(c(u)) \quad \textrm{si $u \in (t,1]$}
\end{eqnarray}

$\hat{c}$ est bien défini, continu, et relève $c$ en partant de $m$ sur $[0,1]$. L'application développante est donc un revêtement. Comme on a supposé que $X$ était simplement connexe, c'est bien un difféomorphisme.
\end{proof}

Dans le cadre riemannien, le théorème de Hopf-Rinow affirme qu'une métrique est géodésiquement complète si et seulement si la distance qu'elle induit est complète. En particulier toute variété riemannienne compacte est géodésiquement complète, et toute variété riemannienne globalement homogène est géodésiquement complète. En combinant ces remarques avec la proposition précédente, on obtient:

\begin{CiteThm}[Ehresman, \cite{Ehresmann36}]
Soit $G$ un groupe de Lie agissant transitivement sur $X$ en préservant une métrique riemannienne. Alors toute $(G,X)$-structure compacte est complète.
\end{CiteThm}

Mais tout ceci tombe en défaut lorsque la métrique n'est plus définie positive. D'une part, une variété pseudo-riemannienne globalement homogène n'est pas nécessairement géodésiquement complète. D'autre part, une variété pseudo-riemannienne compacte peut aussi ne pas être géodésiquement complète. (Des contre-exemples apparaissent dès la dimension 3. Voir par exemple \cite{GuediriLafontaine95}.)

Le premier problème ne se pose plus lorsque l'on impose une condition plus forte que l'homogénéité: la symétrie (au sens de Cartan).

\begin{definition}
Une variété pseudo-riemannienne $X$ est symétrique au sens de Cartan si pour tout point $x\in X$, l'application $-\Id : T_xX \to T_xX$ s'étend via le flot géodésique en une isométrie globale.
\end{definition}

\begin{proposition}
Les variétés pseudo-riemanniennes symétriques sont géodésiquement complètes.
\end{proposition}

\begin{proof}
Soit $X,g$ une variété pseudo-riemannienne symétrique au sens de Cartan.
Soit $\gamma$ une géodésique. Soit $I$ l'intervalle maximal sur lequel $\gamma$ est définie. Montrons que $I = \R$. Soit $t_0\in I$, et soit $s_{t_0}$ la symétrie centrale de centre $\gamma(t_0)$, c'est-à-dire l'unique isométrie globale de $X$ fixant $\gamma(t_0)$ et dont la différentielle en $\gamma(t_0)$ est $-\Id$.
Posons $\eta(t) = s_{t_0}(\gamma (2t_0 - t))$. La courbe $\eta(t)$ est une géodésique définie sur l'intervalle $J = 2t_0 - I$, et on a $\eta(t_0) = \gamma(t_0)$ et $\stackrel{.}{\eta}(t_0) = \gammapoint(t_0)$. Donc $\eta$ prolonge $\gamma$. Par maximalité de $I$, on a $J\subset I$, et $I$ est donc stable par symétrie de centre $t_0$. Comme ceci est vrai pour tout $t_0 \in I$, on a clairement $I = \R$.
\end{proof}

Dans la suite, nous nous intéresserons à deux familles d'espaces pseudo-riemanniens symétriques.

Tout d'abord, les espaces pseudo-riemanniens de courbure constante. Ce sont ceux qui possèdent un groupe d'isométries global ``maximal'' (i.e. de dimension $\frac{n(n+1)}{2}$, où $n$ est la dimension de l'espace). Il en existe une construction explicite (\cite{Wolf74}, p.63) qui généralise celle des espaces riemanniens de courbure constante : $\R^n$, $\mathbb{S}^n$ et $\H^n$. En particulier, les espaces pseudo-riemanniens de courbure nulle sont les espaces affines munis d'une métrique pseudo-riemannienne invariante par translation. Leur groupe d'isométries est le groupe des transformations affines dont la partie linéaire préserve la métrique.

Une autre famille d'exemples est obtenue en considérant un groupe de Lie muni d'une métrique invariante par translations à gauche et à droite, et en particulier un groupe de Lie semi-simple muni de sa métrique de Killing (voir section \ref{GroupeLie}). Dans ces cas-là, les géodésiques partant de l'élément neutre sont les sous-groupes à 1 paramètre.

Nos résultats de complétude pourront donc être compris comme des résultats de complétude géodésique de certaines variétés localement modelées sur des espaces pseudo-riemanniens symétriques.

\subsection{$(G,X)$-structures kleiniennes et domaines divisibles} \label{Kleinienne}
Un obstacle important à la compréhension de la topologie des $(G,X)$-variétés est la potentielle complexité de l'application développante. En dehors des situations où l'on sait qu'elle est nécessairement un difféomorphisme global, elle n'est en général ni surjective, ni injective, ni même un revêtement sur son image. Les contre-exemples les plus connus apparaissent en géométrie projective complexe.

Considérons ainsi une surface compacte $S$ de genre $g\geq 2$ et munissons-la d'une structure complexe. L'uniformisation de cette structure complexe identifie $S$ à un quotient du demi-plan de Poincaré par un réseau de $\PSL(2,\R)$, ce qui est un cas particulier de structure projective complexe appelée structure fuchsienne. D'après le théorème d'Ehresmann-Thurston, si l'on déforme la représentation d'holonomie dans $\PSL(2,\C)$, la nouvelle représentation reste l'holonomie d'une structure projective. Lorsque la déformation est petite, l'application développante reste injective, et son image est un domaine simplement connexe bordé par une courbe de Jordan de dimension de Hausdorff $>1$ (on obtient une structure \emph{quasi-fuchsienne}). Toutefois, si l'on continue à déformer, l'application développante finit par se ``recouper'' et devient très complexe. D'après le théorème de Gallo-Kapovich-Marden (\cite{GalloKapovichMarden00}), presque toutes les représentations de $\pi_1(S)$ dans $\PSL(2,\C)$ qui se relèvent à $\SL(2,\C)$ sont l'holonomie d'une structure projective sur $S$. Beaucoup de ces représentations ne sont pas d'image discrète, et les applications développantes associées sont loin d'être quasi-fuchsiennes.

Pour généraliser la notion de structure quasi-fuchsienne à d'autres géométries, Kularni et Pinkall introduisent dans \cite{KulkarniPinkall85} la terminologie de \emph{structure kleinienne}, que nous reprenons ici.

\begin{definition} 
Soit $M$ une $(G,X)$-variété, et $(\dev,\rho)$ une application développante et un morphisme d'holonomie associés.
Lorsque $\dev$ est un revêtement sur son image, on dit que la $(G,X)$-structure sur $M$ est \emph{quasi-kleinienne}. Elle est dite \emph{kleinienne} si, de plus, le groupe $\pi_1(M)/\ker(\rho)$ agit proprement discontinument et librement sur l'image de $\dev$. Dans ce cas, la $(G,X)$-structure identifie $M$ à un quotient d'un ouvert de $X$ par un sous-groupe discret de $G$ agissant proprement et librement. 
\end{definition}

En particulier, si $\dev$ est injective, la $(G,X)$-structure est kleinienne. Notamment, les $(G,X)$-variétés complètes sont kleiniennes. La notion de structure kleinienne est à rapprocher de celle d'\emph{ouvert divisible}:

\begin{definition}

Soit $G$ un groupe de Lie et $X$ un espace $G$-homogène. Un ouvert $U$ de $X$ est \emph{divisible} s'il existe un sous-groupe discret de $G$ agissant proprement et cocompactement sur $U$. Autrement dit, $U$ est divisible s'il existe un pavage $G$-régulier de $U$.
\end{definition}

Si une variété compacte possède une $(G,X)$-structure kleinienne, l'image de l'application développante est un ouvert de $X$ divisible par $G$. Réciproquement, si un ouvert $U$ est divisible par $G$, soit $\Gamma$ un sous-groupe discret de $G$ agissant proprement et cocompactement sur $U$. Si, de plus, $\Gamma$ est sans torsion, il agit librement sur $U$, et le quotient de $U$ par $\Gamma$ fournit une $(G,X)$-variété kleinienne. Dans tous les exemples que nous considérons, le lemme de Selberg nous garantit que $\Gamma$ possède un sous-groupe d'indice fini sans torsion. Par conséquent, L'existence d'ouverts de $X$ divisibles par l'action de $G$ est équivalente à l'existence de $(G,X)$-variétés compactes kleiniennes.

Une bonne façon de trouver des $(G,X)$-variétés compactes non complètes serait donc de trouver des ouverts divisibles non triviaux. Cependant, la description des ouverts de $X$ divisibles sous l'action de $G$ est un problème difficile. Kozul, puis Benoist ont étudié les convexes de $\R^n$ divisibles par $\PGL(n+1, \R)$, donnant ainsi des exemples de structures projectives réelles exotiques sur les variétés compactes (lire par exemple \cite{BenoistConv4}). Toutefois, pour de nombreuses autres géométries, on ne sait pas s'il existe des ouverts divisibles non triviaux.\\

Nos théorèmes peuvent donc s'interpréter de trois façons différentes: comme des résultats de complétude de certaines $(G,X)$-structures, comme des résultats de complétude géodésique de certaines variétés pseudo-riemanniennes localement homogènes, ou encore comme des résultats d'inexistence d'ouverts divisibles non triviaux dans certains espaces homogènes.

\section{Autour de la conjecture de Markus} \label{Affine} \label{Markus}

Une \emph{structure affine} sur une variété $M$ est simplement une $(\Aff(\R^n), \R^n)$-structure. Comme le groupe affine préserve une connexion linéaire plate et sans torsion sur $\R^n$ (celle pour laquelle les champs invariants par translation sont parallèles), toute variété affine hérite d'une connexion plate et sans torsion. Réciproquement, une connexion plate et sans torsion sur une variété $M$ induit une structure affine. De plus, la structure affine est complète si et seulement si la connexion qu'elle induit est géodésiquement complète (la preuve est identique à celle de la proposition \ref{PropGeodComp}).

Il existe de nombreux exemples de variétés affines compactes non complètes. Les plus simples sont les variétés de Hopf, quotients de $\R^n \backslash \{0\}$ par un sous-groupe monogène d'homothéties linéaires. Le problème de la caractérisation des structures affines compactes complètes a été formulé par Markus.

\begin{conjecture}[Markus, \cite{Markus63}]
Soit $M$ une variété compacte munie d'une structure affine. Si $M$ possède une forme volume parallèle, alors $M$ est complète. Autrement dit, les $(SL(n,\R) \ltimes \R^n, \R^n)$-variétés compactes sont complètes.
\end{conjecture}

Cette conjecture a été attaquée sous plusieurs angles. La première voie, empruntée par Smilie puis Fried, Goldman et Hirsch, consiste à ajouter des hypothèses sur la structure du groupe fondamental de $M$.

\begin{CiteThm}[Fried, Goldman, Hirsch,\cite{FGH81}]
Soit $M$ une variété compacte munie d'une structure affine et possédant une forme volume parallèle. Si le groupe fondamental de $M$ est nilpotent, alors $M$ est complète.
\end{CiteThm}

Ce théorème généralise un théorème de Smillie qui supposait le groupe fondamental de $M$ abélien.

L'autre angle d'attaque a consisté à supposer l'existence d'une structure géométrique parallèle plus contraignante qu'une forme volume. Rappelons que les variétés pseudo-riemanniennes plates sont localement modelées sur $\R^n$ muni de l'action affine de $\O(p,n-p)\ltimes \R^n$, où $(p,n-p)$ est la signature de la métrique. \`A revêtement double près, elles sont orientables, et possèdent alors une forme volume parallèle (car tous les éléments de $\O(p,n-p)$ sont de déterminant $\pm 1$). Par conséquent la conjecture de Markus impliquerait que toutes les variétés pseudo-riemanniennes plates sont géodésiquement complètes. C'est ce que Carrière a prouvé dans le cadre lorentzien.

\begin{CiteThm}[Carrière, \cite{Carriere89}]
Soit $M$ une variété compacte munie d'une métrique lorentzienne plate. Alors $M$ est géodésiquement complète (et donc un quotient de $\R^n$ par un sous-groupe de transformations affines lorentziennes). Autrement dit, les $(\O(n-1,1)\ltimes \R^n, \R^n)$-structures sont complètes.
\end{CiteThm}

Citons enfin un théorème de Jo--Kim concernant les ouverts divisibles:

\begin{CiteThm}[Jo, Kim, \cite{JoKim04}]
Soit $M$ une variété affine compacte de la forme $\Gamma \backslash \Omega$, où $\Omega$ est un ouvert convexe strictement inclus dans $\R^n$ et $\Gamma$ un groupe discret de transformations affines. On peut décomposer $\Omega$ sous la forme $\R^k \times \Omega'$, où $\Omega'$ est un convexe propre de $\R^{n-k}$. Supposons que $\Aut_{proj}(\Omega')$ est irréductible. Alors $M$ ne possède pas de forme volume parallèle.
\end{CiteThm}
Ce théorème est, à notre connaissance, le seul qui traite de la conjecture de Markus dans le cas particulier d'une structure kleinienne.\\

La preuve du théorème de Carrière repose sur l'idée que le groupe des transformations lorentziennes est de \emph{discompacité 1}, c'est-à-dire que les ellipsoïdes ne peuvent dégénérer sous l'action de ce groupe qu'en des hyperplans. 
Nous allons nous intéresser dans la suite à l'exemple le plus simple d'action affine de discompacité supérieure à 2.

\subsection{Métriques lorentz-hermitiennes plates}

Convenons d'appeler forme lorentz-hermitienne sur $\C^n$ une forme sesquilinéaire de signature réelle $(2n-2,2)$. Une métrique lorentz-hermitienne sur une variété complexe $M$ est alors une section lisse de la fibration des formes lorentz-hermitiennes sur le fibré tangent.

Dans toute cette section, $X$ désignera $\C^n$ muni de la métrique lorentz-hermitienne invariante par translations qui s'écrit dans la base canonique

$$\d z_1 \overline{\d z_1} + \ldots + \d z_{n-1} \overline{\d z_{n-1}} - \d z_n \overline{\d z_n} ,$$
et $G = \U(n-1,1) \ltimes \C^n$ le groupe des transformations affines complexes préservant cette métrique. On notera $(e_1, \ldots, e_n)$ la base canonique.

Cette géométrie est le premier exemple de géométrie affine auquel le théorème de Carrière ne s'applique pas. Le groupe $\U(n-1,1)$ est, au sens de Carrière, de discompacité 2. Son action contracte une direction complexe, et les ellipsoïdes dégénèrent donc sur des hyperplans complexes, qui sont de codimension réelle 2.

Le théorème \ref{ThmAffine} constitue donc une avancée par rapport au théorème de Carrière. Il est toutefois moins satisfaisant puisqu'il ne prouve pas que \textit{toutes} les variétés lorentz-hermitiennes plates compactes sont complètes, ce qui serait une conséquence de la conjecture de Markus (puisque $\U(n-1,1)$ est inclus dans $\SL(2n,\R)$). Son intérêt est surtout de présenter dans un cadre plus simple la technique qui nous servira dans la section \ref{GroupeLie}.

\subsection{Preuve du théorème \ref{ThmAffine}}

Considérons donc un ouvert $U$ de $\C^n$ et un sous-groupe $\Gamma$ de $\U(n-1,1)\ltimes \C^n$ agissant proprement, discrètement et cocompactement sur $U$. Soit $l : \U(n-1,1)\ltimes \C^n \to U(n-1,1)$ le morphisme qui à une transformation affine de $\C^n$ associe sa partie linéaire, et notons $\Gamma_0$ l'image de $\Gamma$ par $l$.

Nous voulons montrer que $U = \C^n$. Supposons par l'absurde que $U$ est strictement inclus dans $\C^n$. $U$ possède alors un bord, et le point de départ de notre preuve sera d'observer la dynamique de $\Gamma$ près du bord. La proposition suivante affirme que ce bord contient des ``ellipsoïdes dégénérés sous l'action de $\Gamma$'', qui sont des hyperplans complexes.

\begin{proposition} \label{hyperplan}
Fixons sur $\C^n$ une métrique euclidienne plate préservée par $\U(n-1)\times \U(1)$. Alors il existe un $\epsilon >0$ tel que pour tout $y\in \partial U$, il existe un hyperplan affine complexe $H$ contenant $y$, et tel que $H\cap B(y,\epsilon)\subset \partial U$. ($B(y, \epsilon)$ désigne la boule de centre $y$ de rayon $\epsilon$ pour la métrique euclidienne que nous avons fixée.) 

Ce morceau d'hyperplan inclus dans le bord de $U$ est de plus limite de morceaux d'hyperplans inclus dans $U$; c'est-à-dire qu'il existe une suite $y_n$ convergeant vers $y$, et une suite d'hyperplans affines complexes $H_n$ contenant $y_n$ tels que $H_n$ converge vers $H$ ($H_n$ et $H$ étant vus comme des points dans la grassmannienne de $\C^n$) et $H_n \cap B(y_n, \epsilon) \subset U$.
\end{proposition}

\begin{proof}

Soit $F$ un domaine compact de $U$ tel que $\Gamma \cdot F$ recouvre $U$. Alors pour tout point $y$ dans le bord de $U$, il existe une suite $\gamma_n \in \Gamma$ telle que $\gamma_n \cdot F$ ``s'accumule'' sur $y$, dans le sens où il existe une suite $x_n$ dans $F$ telle que $\gamma_n \cdot x_n$ converge vers $y$. En particulier, aucune sous-suite de $(\gamma_n \cdot F)$ ne reste à l'intérieur d'un compact de $U$, ce qui implique que $\gamma_n$ tend vers l'infini dans $\Gamma$.

Le lemme suivant établit qu'en fait, toute l'adhérence de la suite $(\gamma_n \cdot F)$ est dans le bord de $U$.

\begin{lemme}\label{DgnBord}
Soit $X$ un espace $G$ homogène, $U$ un ouvert de $X$ et $\Gamma$ un sous-groupe discret de $G$ préservant $U$ et agissant proprement sur $U$.
Soit $\gamma_n$ une suite d'éléments de $\Gamma$ tendant vers l'infini, et $x_n$ une suite de points de $U$ inclus dans un compact de $U$. Alors toute valeur d'adhérence de $\gamma_n \cdot x_n$ est dans $\partial U$.
\end{lemme}
\begin{proof}
En effet, soit $y$ une valeur d'adhérence de $\gamma_n \cdot x_n$. Clairement, $y\in \overline{U}$. Si $y$ était dans $U$, on aurait une sous-suite $\gamma_{n_k}$ et un compact $F'$ tel que $\gamma_{n_k} \cdot F$ intersecte $F'$. C'est absurde car $\gamma_n$ tend vers l'infini et que l'action de $\Gamma$ sur $U$ est propre.
\end{proof}

Bien sûr, en général, l'adhérence de $\gamma_n \cdot F$ pourrait se limiter à un point. Mais ce n'est pas le cas ici, où l'action de $U(n-1,1)$ ne contracte qu'une direction complexe. Pour préciser cela, écrivons la décomposition de Cartan de $l(\gamma_n)$:

$l(\gamma_n)$ peut s'écrire  $k_n a_n k'_n$, où $k_n, k'_n \in \U(n-1)\times \U(1)$ et où $a_n$ s'écrit dans la base $(e_1 + e_n, e_2, \ldots, e_{n-1}, e_1 - e_n)$:

$$
\left(
\begin{array}{ccccc}
\lambda_n & 0 & \ldots & \ & 0 \\
0 & 1 & \ & \ & \ \\
\vdots & \ & \ddots & \ & \vdots \\
\ & \ & \ & 1 & 0 \\
0 & \ & \ldots & 0 & 1/\lambda_n
\end{array}
\right)
$$
avec $\lambda_n$ réel supérieur à $1$. 

Soit $\epsilon > 0$ tel que l'$\epsilon$-voisinage de $F$ est relativement compact dans $U$. Notons $B_\epsilon$ la boule de centre $0$ de rayon $\epsilon$. On sait que $x_n + B_\epsilon \subset U$, et donc que $\gamma_n \cdot \left( x_n + B_\epsilon \right) \subset U$.
Or
$$\gamma_n \cdot \left( x_n + B_\epsilon \right) = y_n + l(\gamma_n) B_\epsilon.$$

Soit $H_0$ l'hyperplan $\Vect^\C (e_1+e_n, e_2, \ldots, e_{n-1})$. La matrice diagonale $a_n$ dilate $H_0$. On en déduit que 
$l(\gamma_n) B_\epsilon$ contient $(k_n H_0) \cap B_\epsilon$.
Posons $H_n = k_n H_0$. Quitte à extraire, on peut supposer que $H_n$ converge vers un hyperplan vectoriel $H$. Alors $U$ contient $y_n + (H_n \cap B_\epsilon)$, et $\overline U$ contient la limite $y + (H\cap B_\epsilon)$. Enfin, $y + (H\cap B_\epsilon)$ est inclus dans $\partial U$, d'après le lemme \ref{DgnBord}.

\end{proof}

\begin{proposition} 
Soit $y\in \partial U$. Alors il existe un unique hyperplan affine complexe $H$ contenant $y$ tel que $H \cap \partial U$ est un ouvert de $H$.
\end{proposition}

\begin{proof}
D'après la proposition \ref{hyperplan}, un tel $H$ existe.
Soit donc $H$ un hyperplan affine complexe contenant $y$ et tel que $H \cap B(y, \epsilon') \subset \partial U$. Soit d'autre par une suite $y_n \in U$ convergeant vers $y$, une suite $H_n$ d'hyperplans affines complexes contenant $y_n$ tels que $H_n \cap B(y_n, \epsilon) \subset U$, et supposons que $H_n$ converge vers $H'$, qui contient donc $y$. Supposons que $H'$ est différent de $H$. Alors $H'$ et $H$ sont transverses en $y$ (ce sont deux hyperplans complexes), et pour $n$ assez grand, $H_n \cap B(y_n, \epsilon)$ intersecte $H \cap B(y, \epsilon')$, ce qui est absurde puisque $H \cap B(y, \epsilon') \subset \partial U$ alors que $H_n \cap B(y_n, \epsilon) \subset U$. Donc $H = H'$. Par conséquent, $H$ est unique. 
\end{proof}

\begin{proposition}
L'ouvert $U$ est feuilleté par des copies parallèles d'un même hyperplan complexe, et $\Gamma_0$ préserve l'hyperplan vectoriel sous-jacent.
\end{proposition}

\begin{proof}
Soit $y\in \partial U$, et $H$ l'unique hyperplan contenant $y$ tel que $H\cap B(y,\epsilon) \subset \partial U$. Montrons que tout $H$ est inclus dans $\partial U$. Soit $A$ l'ensemble des $z \in H$ tels que $H\cap B(z,\epsilon) \subset \partial U$. $A$ est clairement fermé, et contient $y$. Soit $z \in A$ et $z' \in H\cap B(z, \epsilon/2)$. Alors $H\cap B(z', \epsilon/2) \subset \partial U$. D'après la proposition précédente, $H$ est l'unique hyperplan contenant $z'$ et tel que $H\cap \partial U$ est ouvert, et par conséquent, $H\cap B(z', \epsilon) \subset \partial U$.
D'où $B(z, \epsilon/2) \subset A$. $A$ est ouvert et fermé, donc $A = H$ et $H\subset \partial U$.

Soient $H$ et $H'$ deux hyperplans affines contenus dans $\partial U$. S'ils ne sont pas parallèles, ils s'intersectent transversalement en un point $y\in \partial U$, ce qui contredit l'unicité de l'hyperplan contenu dans $\partial U$ et passant par $y$. Donc $H$ et $H'$ sont parallèles.

Soit $H_0$ la partie linéaire des hyperplans contenus dans le bord de $U$. Il est clair que $\Gamma_0$ préserve $H_0$. Soit maintenant $x \in U$. Supposons que $H_0 +x$ n'est pas inclus dans $U$. Alors il existe $y \in (H_0+x) \cap \partial U$. Mais alors $(H_0 + y) \in \partial U$, et en particulier $x\in \partial U$, ce qui est absurde. Donc $U$ est feuilleté par des copies parallèles de $H_0$.
\end{proof}

L'ouvert $U$ fibre donc au dessus d'un ouvert $\Omega$ de $H_0 \backslash \C^n \sim \C$, et l'action de $\Gamma$ sur $\C^n$ induit une action sur $\C$ par transformations affines complexes, qui préserve $\Omega$. On veut prouver que $\Omega = \C$.

Notons $\overline{\rho}$ la représentation de $\Gamma$ dans $\Aff(\C)$ induite par passage au quotient par $H_0$. Remarquons le fait suivant :

\begin{lemme}
Si l'image de $\Gamma$ par $\overline{\rho}$ est discrète, alors elle est virtuellement abélienne.
\end{lemme}

\begin{proof}
Cela est vrai de tous les sous-groupes de $\Aff(\C)$, mais, par commodité, nous le prouvons uniquement pour des sous-groupes de type fini. Remarquons que $\overline{\rho}(\Gamma)$ est de type fini, puisque c'est l'image du groupe fondamental d'une variété compacte. 

Soit $N$ un sous-groupe de type fini de $\Aff(\C)$. Supposons que $N$ n'est pas virtuellement abélien. $N$ n'est donc pas virtuellement inclus dans le sous-groupe des translations. La partie linéaire $N_0$ de $N$ est donc infinie. Si $N_0$ est inclus dans le sous-groupe des rotations, comme $N_0$ est de type fini, il est monogène. Il existe donc une rotation affine dans $N$ d'angle irrationnel, et $N$ n'est donc pas discret.
Si $N_0$ n'est pas inclus dans le sous-groupe des rotations, il existe un élément $n$ de $N$ de la forme $z \mapsto az + b$, avec $|a| < 1$. Quitte à conjuguer le groupe $N$, on peut supposer que le point fixe de $n$ est $0$ (autrement dit, que $b=0$).
Comme $N$ n'est pas abélien, $N$ contient un élément $n'$ qui n'a pas $0$ comme point fixe (i.e. $n' : z\mapsto a'z+b', \ b'\neq 0$). Alors $n^k n' n^{-k}$ envoie $z$ sur $a' z + a^k b$. La suite $n^k n' n^{-k}$ converge sans être stationnaire, et $\Gamma$ n'est donc pas discret.
\end{proof}

Montrons que si $\Omega \neq \C$, $\overline{\rho}(\Gamma)$ ne peut pas être discret. Si $\overline{\rho}(\Gamma)$ est discret, il est virtuellement abélien. S'il est virtuellement inclus dans le sous-groupe des translations, $\Omega/\overline{\rho}(\Gamma)$ est, à revêtement fini près, un ouvert d'un tore. Comme $\Omega/\overline{\rho}(\Gamma)$ doit être compact, $\Omega = \C$.
Sinon, $\overline{\rho}(\Gamma)$ est virtuellement inclus dans le stabilisateur d'un point, et stabilise donc une orbite finie. Mais alors, $\Gamma$ stabilise une famille finie d'hyperplans affines complexes de $\C^n$, et fixe donc virtuellement un hyperplan affine complexe. Cela est impossible d'après le théorème suivant :

\begin{CiteThm}[Goldman, Hirsch, \cite{GoldmanHirsch84}]
Soit $M$ une variété compacte munie d'une structure affine dont l'holonomie préserve un sous-espace affine strict. Alors $M$ ne possède pas de forme volume parallèle.
\end{CiteThm}

Nous savons donc que l'adhérence de $\overline{\rho}(\Gamma)$ dans $\Aff(\C)$ contient un sous-groupe à 1 paramètre. Ce sous-groupe à 1 paramètre stabilise le bord de $\Omega$, qui est donc une réunion d'orbites de ce sous-groupe. Les sous-groupes à 1 paramètre du groupe affine sont les rotations fixant un point, les homothéties fixant un point, les ``homothéties-rotations'' fixant un point ($z\to e^{t(a+ib)} z +c$), et les translations le long d'une direction. $\partial \Omega$ est donc soit un point, soit une réunion de cercles concentriques, de droites parallèles, de demi-droites ayant même origine, où de spirales logarithmiques s'enroulant autour d'un même point. Sachant que $\overline{\rho}(\Gamma)$ doit agir cocompactement sur l'une des composantes connexes du complémentaire de $\partial \Omega$, on se convainc facilement que $\partial \Omega$ ne peut être qu'un point ou une droite (dans tous les autres cas, le stabilisateur du bord de $\Omega$ est exactement le sous-groupe à 1 paramètre, et il n'agit pas cocompactement). Mais alors $\Gamma$ stabilise soit un hyperplan complexe, soit un hyperplan réel de $\C^n$, ce qui, encore une fois, contredit le théorème de Goldman et Hirsch.

On conclut donc que $\Omega = \C$, et que $U = \C^n$.

\begin{rmq}
La preuve ne repose que sur le fait que $\U(n-1,1)$ est en quelque sorte de \emph{discompacité complexe 1}, dans le sens où son action fait dégénérer les ellipsoïdes sur des hyperplans complexes. La même preuve permet donc d'obtenir le théorème \ref{ThmHoloPlat}. En effet, tous les éléments de $\SO(3,\C)$ se décomposent sous la forme $k a k'$, où $k \in \SO(3,\R)$ et où $a$ est conjugué à la matrice
$$
\left( \begin{array}{ccc}
\lambda & 0 & 0 \\
0 & 1 & 0 \\
0 & 0 & \lambda^{-1}
\end{array} \right) \quad, $$
avec $\lambda$ réel supérieur ou égal à $1$.
\end{rmq}

\section{Groupes de Lie de rang 1 et leur métrique de Killing} \label{GroupeLie}
 
Nous nous intéressons dans cette section à une autre famille de géométries. Considérons un groupe de Lie semi-simple $L$, connexe et simplement connexe, d'algèbre de Lie $\l$. La forme de Killing est la forme quadratique définie sur $\l$ par
$$\kappa_\l(u, v) = \Tr(\ad_u \ad_v).$$
La forme de Killing est non dégénérée, de signature $(\dim L - \dim K, \dim K)$ où $K$ est un sous-groupe compact maximal de $L$, et préservée par l'action adjointe de $L$. On appelle \emph{métrique de Killing} sur $L$ (et on note $\kappa_L$) la métrique pseudo-riemannienne obtenue en étendant $\kappa_\l$ à $L$ par invariance à gauche. Comme $\kappa_\l$ est invariante par l'action adjointe, la métrique de Killing est aussi invariante à droite. Par conséquent, l'application $g \mapsto g^{-1}$ préserve $\kappa_L$. On en déduit aisément que $(L,\kappa_L)$ est un espace pseudo-riemannien symétrique au sens de Cartan. La composante connexe de l'identité dans son groupe d'isométries s'identifie à $L\times L$ (à revêtement près).

Une variété $M$ munie d'une métrique pseudo-riemannienne localement isométrique à la métrique de Killing de $L$ possèdera donc une $(L\times L, L)$-structure, et le théorème \ref{ThmGrpLie} concerne ces variétés pseudo-riemanniennes. Dans la suite, pour éviter la terminologie un peut lourde de ``$(L\times L, L)$-structure'', on posera $G = L\times L$.

\subsection{Exemples}
Le plus petit groupe de Lie simple non compact est $\PSL(2,\R)$, de dimension 3. Sa métrique de Killing est lorentzienne, et, comme son groupe d'isométries est de dimension 6, elle est de courbure constante. Bien qu'il ne soit pas simplement connexe, $\PSL(2,\R)$ muni de sa métrique de Killing est appelé \emph{espace anti-de Sitter} de dimension 3 (analogue lorentzien de l'espace hyperbolique de dimension 3). Dans ce cas, le problème de la complétude est résolu par le théorème de Klingler, qui généralise le théorème de Carrière.

\begin{CiteThm}[Klingler, \cite{Klingler96}]
Soit $M$ une variété compacte munie d'une métrique lorentzienne de courbure constante. Alors $M$ est géodésiquement complète.
\end{CiteThm}

Les variétés anti-de Sitter compactes de dimension 3 sont donc des quotients de $\widetilde{\PSL(2,\R)}$ par un sous-groupe discret de $\widetilde{\PSL(2,\R)} \times \widetilde{\PSL(2,\R)}$ agissant proprement, librement et cocompactement sur $\widetilde{\PSL(2,\R)}$. D'après un théorème de Kulkarni et Raymond (\cite{KulkarniRaymond85}), ce sont en fait des quotients d'un revêtement fini de $\PSL(2,\R)$.\\

$\PSL(2,\C)$ est également un groupe de Lie simple de rang 1. Vu comme groupe de Lie réel, il est isomorphe à la composante neutre de $\PSO(3,1)$, de dimension 6, et sa métrique de Killing est de signature (3,3). Mais on peut aussi définir sur $\PSL(2,\C)$ une métrique de Killing complexe, donnée dans l'algèbre de Lie tangente par $$K^\C(u,v) = \Tr^\C(\ad_u \ad_v).$$
Cette métrique est une \emph{métrique riemannienne-holomorphe} (i.e. une section holomorphe partout non-dégénérée du fibré $\Sym^2 \T^*\PSL(2,\C)$). La notion de courbure se généralise pour de telles métriques, et celle-ci est de courbure constante non nulle.
$\PSL(2,\C)$ muni de sa métrique de Killing est donc un analogue complexe de l'espace anti-de Sitter de dimension 3. Toutefois, la preuve de Klingler ne se généralise pas dans ce cas-là. Elle repose en effet, comme le théorème de Carrière, sur le fait que les isométries lorentziennes sont de discompacité 1. Or, l'action adjointe de $\PSL(2,\C)$ sur lui-même est de discompacité 2.

\subsection{$(G,L)$-structures complètes standards et non-standards} \label{Kassel}

Dans toute la suite, $L$ désigne un groupe de Lie semi-simple de rang 1. On sait, depuis les travaux de Borel (\cite{Borel63}), qu'un tel groupe admet toujours un réseau cocompact $\Gamma$. Le quotient de $L$ par l'action de $\Gamma$ à droite fournit donc un exemple de $(G,L)$-variété compacte et complète. Une telle structure est appelée standard\footnote{Nous utilisons ce terme dans un sens un peu différent de celui de Kulkarni et Raymond dans \cite{KulkarniRaymond85}.}.

D'après le théorème de rigidité de Mostow, si $\l$ n'a pas de facteur isomorphe à $\sl(2,\R)$, toute déformation de $\Gamma$ dans $L$ est obtenue par conjugaison, et ne modifie pas la $(G,L)$-structure. En revanche, lorsque $L = \SO(n,1)$ ou $\SU(n,1)$, on peut déformer $\Gamma$ de façon non triviale dans $L\times L$ (\cite{Kobayashi98}). Plus précisément, tout morphisme $u : \Gamma \to L$ fournit un plongement $\rho_u : \Gamma \to L\times L$ défini par
$$\rho_u : \gamma \mapsto (u(\gamma), \gamma).$$

D'après le principe d'Ehresman-Thurston, pour $u$ suffisamment proche du morphisme trivial, $\rho_u$ est l'holonomie d'une $(G,L)$-structure sur la variété $M$. Goldman, dans le cas de $\SL(2,\R)$ (\cite{Goldman85}), Ghys, pour $\SL(2,\C)$ (\cite{Ghys95}), et Kobayashi dans le cas général (\cite{Kobayashi98}) ont remarqué que, pour $u$ petit, $\rho_u(\Gamma)$ continue d'agir proprement, discrètement et cocompactement sur $L$. Autrement dit, lorsqu'on déforme un peu une $(G,L)$-structure complète standard, on obtient toujours une $(G,L)$-structure complète, mais cette fois-ci non-standard.

Dans la lignée des travaux de Ghys (\cite{Ghys95}), Kulkarni et Raymond (\cite{KulkarniRaymond85}), Kobayashi (\cite{Kobayashi93}, \cite{Kobayashi98}), Salein (\cite{Salein00}), les travaux de Kassel, Guéritaud, Guichard et Wienhard décrivent très précisément les $(G,L)$-structures complètes.

\begin{CiteThm}[Kassel, \cite{Kassel08}]
Soit $L$ un groupe de Lie semi-simple de rang $1$, et $\Gamma$ un sous-groupe de $L\times L$ agissant proprement, discrètement et cocompactement sur $L$. Notons $\rho_1$ et $\rho_2$ les projections de $\Gamma$ sur chaque coordonnée. Alors il existe $i\in \{1,2\}$ tel que $\rho_i$ est injective et $\rho_i(\Gamma)$ est un réseau cocompact de $L$.
\end{CiteThm}
Autrement dit, tous les quotients compacts de $L$ par un sous-groupe de $L\times L$ sont de la forme $\rho_u(\Gamma)\backslash L$, où $\Gamma$ est un réseau de $L$ et $\rho_u$ est un morphisme comme décrit précédemment. Ce théorème précise un résultat de Kobayashi (\cite{Kobayashi93}) qui affirmait que l'une des projections était injective.\\

Réciproquement, Kassel donne dans \cite{KasselThese}, pour le cas particulier où $L = \SL(2,\R)$, un critère pour que la représentation $\rho_u$ agisse proprement et discrètement. Ce critère est étendu par Guéritaud et Kassel aux groupes $\SO(n,1)$ (\cite{GueritaudKassel}), puis par Guéritaud, Kassel, Guichard et Wienhard à tous les groupes de Lie semi-simples (\cite{GGKW}).

\begin{CiteThm}[Guéritaud, Guichard, Kassel, Wienhard]
Soit $L$ un groupe de Lie semi-simple de rang $1$, $K$ un sous-groupe compact maximal. $L$ agit sur $L/K$ en préservant une métrique riemannienne symétrique. Soit $\Gamma$ un réseau dans $L$ et $u$ un morphisme de $\Gamma$ dans $L$. Alors $\rho_u(\Gamma)$ agit proprement discontinument sur $L$ si et seulement s'il existe un point $x\in L/K$ et deux constantes $C$ et $C'$, $C<1$ telles que 
$$\forall \gamma \in \Gamma, \d(x, u(\gamma) \cdot x) \leq C \d(x, \gamma \cdot x) + C'.$$
\end{CiteThm}

\begin{CiteCoro}[Guéritaud, Guichard, Kassel, Wienhard]
Soit $L$ un groupe de Lie semi-simple de rang $1$, et $M$ une variété compacte de même dimension que $L$. Alors, dans l'espace $\mathrm{Def}_{(G,L)}(M)$, le domaine des $(G,L)$-structures complètes forme un ouvert.
\end{CiteCoro}

Ces théorèmes fournissent une relativement bonne description des $(G,L)$-structures complètes. En revanche, hormis dans le cas de $\SL(2,\R)$, on ne sait pas s'il existe des $(G,L)$-variétés compactes non complètes. Le théorème \ref{ThmGrpLie} donne une réponse qui, bien que partielle, permet de compléter la description des $(G,L)$-structures complètes. Le théorème \ref{ThmGrpLie} a en effet pour conséquence:

\begin{corollaire} \label{Fermeture}
Soit $L$ un groupe de Lie semi-simple de rang 1, et $M$ une variété compacte de même dimension que $L$. Alors, dans l'espace $\mathrm{Def}_{(G,L)}(M)$, le domaine des $(G,L)$-structures complètes forme un fermé.
\end{corollaire}

\begin{proof}
D'après le théorème \ref{ThmGrpLie}, il suffit de prouver qu'une limite de structures complètes est kleinienne. Considérons une suite de couples $(\dev_n, \rho_n)$ convergeant vers $(\dev, \rho)$, et telle que $\dev_n$ est un difféomorphisme de $\tilde{M}$ dans $L$. Alors en particulier $\dev_n$ est injective. Montrons que cette propriété passe à la limite.

Supposons par l'absurde qu'il existe deux points $x_1$ et $x_2$ distincts de $\tilde{M}$ tels que $\dev(x_1) = \dev(x_2) = y$. Alors, comme $\dev$ est un difféomorphisme local, il existe un voisinage $V$ de $y$ relativement compact, et deux voisinages relativement compacts $U_1$ et $U_2$ de $x_1$ et $x_2$ disjoints tels que $\dev(U_1) = \dev(U_2) = V$. Pour $n$ assez grand, $\dev_n$ est très proche de $\dev$ sur $U_1$ et $U_2$, et $\dev_n(U_1)$ et $\dev_n(U_2)$ continuent à s'intersecter, ce qui contredit l'injectivité de $\dev_n$.

L'application développante à la limite est donc injective, et la $(G,L)$-structure correspondante est kleinienne. D'après le théorème \ref{ThmGrpLie}, elle est donc complète.

\end{proof}

D'après le théorème précédent, l'ensemble des $(G,L)$-structures complètes est un ouvert. D'après ce corollaire, c'est aussi un fermé. On en déduit le corollaire \ref{CoroConnexe}.

\begin{rmq}
Nos théorèmes, ainsi que le corollaire \ref{Fermeture}, sont indépendants des résultats de Kassel et co-auteurs. Seul le corollaire \ref{CoroConnexe} utilise conjointement le théorème \ref{ThmGrpLie} et la caractérisation des structures complètes.
\end{rmq}

\subsection{Preuve du théorème \ref{ThmGrpLie}}

La stratégie de la preuve du théorème \ref{ThmGrpLie} est similaire à celle du théorème \ref{ThmAffine}. On raisonne par l'absurde et on considère $U$ un ouvert strict de $L$, muni d'une action discrète, propre, libre et cocompacte d'un sous-groupe $\Gamma$ de $G = L\times L$. Ce domaine $U$ possède un bord non vide $\partial U$, où l'action de $\Gamma$ n'est plus propre. 

Comme dans la preuve du théorème \ref{ThmAffine}, on déduit du lemme \ref{DgnBord} que le bord de $U$ contient des compacts dégénérés sous l'action de $\Gamma$. Le recours à la \emph{décomposition de Cartan} de $L$ permet de prouver que ces compacts dégénèrent sur des morceaux de sous-espaces paraboliques, qui joueront le rôle que jouaient précédemment les hyperplans complexes.

Nous prouverons ensuite que ces sous-espaces paraboliques contenus dans le bord sont en fait des translations à droite (ou à gauche) d'un même sous-groupe parabolique, et nous en déduirons que l'ouvert $U$ fibre au dessus d'un domaine de $L/P$. Le point clé sera de montrer qu'à l'instar de deux hyperplans complexes dans $\C^n$, deux sous-espaces paraboliques qui ne sont pas parallèles s'intersectent toujours transversalement. Nous aurons besoin pour cela de rappeler quelques résultats classiques sur la structure des sous-groupes paraboliques d'un groupe de Lie de rang 1.

Pour finir, nous montrerons que le complémentaire de ce domaine de $L/P$ doit contenir au moins deux points, puis que l'action de $\Gamma$ ne peut pas être discrète, ce qui contredira le fait qu'elle soit totalement discontinue sur un ouvert.\\

\subsubsection{Décomposition de Cartan}

Rappelons que pour tout sous-groupe compact maximal $K$ de $L$, il existe une involution de Cartan $\theta$ préservée par $K$, c'est-à-dire un morphisme d'algèbre de Lie de $\l$ dans $\l$ tel que $\theta^2 = \Id$ et tel que $\theta$ commute avec l'action adjointe de $K$. La forme bilinéaire $B_\theta (., .)= K_\l(., -\theta .)$ est définie positive, et s'étend donc en une métrique riemannienne invariante à droite sur $L$, qui est également préservée par l'action à gauche de $K$. Dans toute la suite, lorsque nous utilisons des notions métriques sur $L$ ou $\l$, nous nous référons implicitement à la métrique associée à une involution de Cartan fixée, préservée par un sous-groupe compact maximal $K$ fixé. Bien qu'il ne soit pas toujours nécessaire d'avoir recours à une métrique aussi spécifique, cela simplifiera grandement les énoncés.

Comme $L$ est de rang 1, il existe un élément $a\in \k^\bot$ dont l'action adjointe sur $\l$ est diagonalisable, et tel que tout élément $g$ de $L$ se décompose sous la forme 
$$g=k_1 \exp(ta) k_2,$$
avec $t \geq 0$ et $k_1, k_2\in K$.
De plus, un tel $t$ est unique. Il s'agit, dans le cas particulier du rang 1, de la décomposition de Cartan de l'élément $g$ (pour plus de précisions sur la structure des groupes de Lie semi-simples, on consultera par exemple \cite{Knapp96}).

En outre, la diagonalisation de $\ad_a$ donne une décomposition de $\l$ sous la forme

$$\l = \m \oplus \R a \oplus \n_+ \oplus \n_-,$$

où $\n_+$ (resp. $\n_-$) est la somme des sous-espaces propres de $\ad_a$ correspondant à des valeurs propres strictement positives (resp. strictement négatives), et où $\m = \k \cap \com(a)$.

Alors, $\p_0 = \m \oplus \R a \oplus \n_+$ forme une sous-algèbre de Lie résoluble de $\l$. Les sous-algèbres de Lie conjuguées à $\p_0$ par l'action adjointe sont appelées sous-algèbres de Lie paraboliques
\footnote{Cette définition \emph{ad hoc} est plutôt celle d'une sous-algèbre de Lie parabolique minimale. Mais en rang 1, toutes les sous-algèbres paraboliques sont minimales et conjuguées.}.
Le stabilisateur sous l'action adjointe d'une sous-algèbre de Lie parabolique $\p$ est un sous-groupe de Lie fermé de $L$, d'algèbre de Lie tangente $\p$, appelé sous-groupe parabolique. Nous appellerons \emph{sous-espace parabolique} de $L$ une sous-variété de la forme $Pb$, où $P$ est un sous-groupe parabolique. Enfin, nous appellerons \emph{disque parabolique de rayon $\epsilon$} une sous-variété de la forme
$$D_\p(y,\epsilon) = \exp(\p \cap B_\epsilon) y$$
où $B_\epsilon$ désigne la boule fermée de centre $0$ de rayon $\epsilon$ dans $\l$ (pour le produit scalaire $B_\theta$).

\begin{lemme} \label{DisqueParabolique}
Il existe $\epsilon>0$ tel que pour tout point $y\in \partial U$, le bord de $U$ contient un disque parabolique centré en $y$ de rayon $\epsilon$. Ce disque est de plus une limite de disques paraboliques contenus dans $U$, dans le sens où il existe une suite $y_n$ de points de $U$ convergeant vers $y$, et une suite $\p_n$ de sous-algèbres paraboliques de $\l$ convergeant vers $\p$ (vues comme points dans la grassmannienne de $\l$) telles que
$$\textrm{pour tout } n , D_{\p_n}(y_n, \epsilon) \subset U$$
$$\textrm{et} \quad D_\p(y,\epsilon) \subset \partial U.$$
\end{lemme}

\begin{proof}
Puisque l'action de $\Gamma$ sur $U$ est cocompacte, fixons $F$ un compact de $U$ tel que $\Gamma \cdot F$ recouvre $U$, et soit $\epsilon>0$ tel que $\exp(B_\epsilon) F$ est relativement compact dans $U$.
Soit $y$ un point du bord de $U$, $y_n$ une suite de points de $U$ convergeant vers $y$, et $x_n \in F$, $\gamma_n \in \Gamma$ tels que $\gamma_n \cdot x_n = y_n$.

Par hypothèse sur $\epsilon$, $\exp(B_\epsilon)x_n \subset U$; et comme $U$ est stable par l'action de $\Gamma$, $\gamma_n \cdot \exp(B_\epsilon)x_n \subset U$. Montrons que $\gamma_n \cdot \exp(B_\epsilon)x_n$ contient un disque parabolique de rayon $\epsilon$ centré en $y_n$.

Notons $\gamma_n = (\gamma_{1,n}, \gamma_{2,n})$. On vérifie aisément que

$$\gamma_n \cdot \left(\exp(B_\epsilon)x_n \right) = \exp \left( \Ad_{\gamma_{1,n}} B_\epsilon \right) y_n.$$
\'Ecrivons maintenant la décomposition de Cartan de $\gamma_{1,n}$ :
$$\gamma_{1,n} = k_n \exp(t_n a) k'_n,$$
avec $t_n\geq 0$, $k_n,k'_n \in K$.

Tout d'abord l'action de $\Ad_{k'_n}$ stabilise $B_\epsilon$. Donc 
$$\Ad_{\gamma_{1,n}} B_\epsilon = \Ad_{k_n} \Ad_{\exp(t_n a)} B_\epsilon.$$

Notons comme précédemment $\p_0 = \m \oplus \R a \oplus \n_+$. $\p_0$ est la somme des espaces propres de $\ad_a$ associés aux valeurs propres positives. Par conséquent, $\Ad_{\exp(t_n a)}$ stabilise $\p_0$, et pour tout $u\in \p_0$, 
$$\norm{\Ad_{\exp(t_n a)} u } \geq \norm{u}.$$

On a donc

$$\Ad_{\exp(t_n a)} B_\epsilon \supset \Ad_{\exp(t_n a)} (B_\epsilon\cap \p_0) \supset B_\epsilon\cap \p_0.$$

Enfin, $\Ad_{k_n}(B_\epsilon\cap \p_0) = B_\epsilon \cap \Ad_{k_n} \p_0$, car $\Ad_{k_n}$ préserve la norme que nous avons mise sur $\l$.

En définitive, si l'on note $\p_n = \Ad_{k_n} \p_0$, on a bien
$$D_{\p_n}(y_n, \epsilon) \subset \gamma_n \cdot (\exp(B_\epsilon) x_n) \subset U.$$

Comme tous les $\p_n$ sont conjugués par des éléments de $K$, on peut, quitte à extraire, supposer que $\p_n$ converge dans la grassmannienne de $\l$ vers une sous-algèbre parabolique $\p$.
Enfin, tels que nous les avons construits, tous les disques paraboliques $D_{\p_n}(y_n, \epsilon)$ sont dans $\gamma_n \cdot \exp(B_\epsilon) F$. D'après la proposition \ref{DgnBord}, on a donc $D_{\p}(y, \epsilon) \subset \partial U$.
\end{proof}

Nous avons vu que par tout point du bord passe un disque parabolique. Nous allons ensuite prouver que ce disque est unique et en déduire que l'ouvert $U$ est feuilleté par des sous-espaces paraboliques parallèles. Rappelons d'abord quelques propriétés des sous-groupes et sous-algèbres paraboliques.

\subsubsection{Sous-algèbres paraboliques et sous-groupes paraboliques}

Soit $\p$ une sous-algèbre parabolique de $\l$. Notons $P$ le stabilisateur de $\p$ sous l'action adjointe. Rappelons que $\p$ est l'algèbre de Lie tangente à $P$. Commençons par établir le lemme suivant :

\begin{lemme}
Si $\l$ est de rang $1$ et sans facteur isomorphe à $\sl(2,\R)$, alors $P$ est connexe.
\end{lemme}

\begin{proof}
Soit $P_0$ la composante connexe de l'élément neutre dans $P$. Alors $P_0$ est un sous groupe distingué de $P$, et la projection $\pi : L/P_0 \to L/P$ est un revêtement galoisien de groupe de Galois $P/P_0$. Or il est connu que, lorsque $L$ est de rang 1, la variété $L/P$ est simplement connexe, sauf si $L$ est, à revêtement près, $\SL(2,\R)$. En effet, $L/P$ est le bord de l'espace symétrique $L/K$. Il est par conséquent homéomorphe à une sphère, de dimension $\geq 2$ sauf lorsque $\l = \sl(2,\R)$. Par conséquent, le groupe $P/P_0$ est trivial, et $P$ est donc connexe.
\end{proof}

Comme $L$ est de rang 1, le groupe de Weyl de $L$ n'a que deux éléments. Soit $w$ l'unique élément non trivial du groupe de Weyl.
On a la fameuse décomposition de Bruhat (\cite{Knapp96}):
$$L = P \sqcup PwP.$$

De cette décomposition nous déduisons deux choses.
Tout d'abord :
\begin{proposition}
L'action adjointe est $2$-transitive sur les sous-algèbres paraboliques.
\end{proposition}

\begin{corollaire}
Deux sous-groupes paraboliques qui ne sont pas identiques sont transverses en l'élément neutre.
\end{corollaire}

\begin{proof}
L'espace des sous-algèbres paraboliques s'identifie à $L/P$, où $P$ est un sous-groupe parabolique. Montrons donc que l'action de $L$ sur $L/P$ est $2$-transitive.
Soient $xP$ et $yP$ deux points distincts de $L/P$. La multiplication à gauche par $x^{-1}$ les envoie respectivement sur $P$ et $x^{-1}yP$. Comme $xP$ et $yP$ sont distincts, $x^{-1}y \notin P$. D'après la décomposition de Bruhat, il existe donc $p\in P$ tel que $x^{-1}yP = pwP$. La multiplication à gauche par $p^{-1}$ fixe $P$ et envoie $x^{-1}yP$ sur $wP$. Nous avons donc montré que deux points distincts de $L/P$ peuvent toujours être envoyés par l'action à gauche de $L$ sur $P$ et $wP$.

Prouvons maintenant le corollaire.
Soient $P_1$ et $P_2$ deux sous-groupes paraboliques distincts, d'algèbres de Lie respectives $\p_1$ et $\p_2$. D'après la proposition précédente, $\p_1$ et $\p_2$ peuvent être envoyées par l'action adjointe d'un élément $g$ sur $\m \oplus \R a \oplus \n_+$ et $\m \oplus \R a \oplus \n_-$, dont la somme engendre $\l$. On a donc également $\p_1 + \p_2 = \l$, et $P_1$ et $P_2$ sont donc transverses.
\end{proof}

Ensuite :
\begin{lemme}
Soient $P$ et $P'$ deux sous-groupes paraboliques, et $a$ et $a'$ deux éléments de $L$. Alors $aP$ et $a'P'$ s'intersectent, sauf si $P = P'$ ou si $aPa^{-1} = a'P'a'^{-1}$. Autrement dit, deux sous-espaces paraboliques s'intersectent transversalement sauf si l'un est l'image de l'autre par une multiplication à gauche ou à droite.
\end{lemme}

On dira que deux sous-espaces paraboliques sont parallèles à gauche (resp. à droite) lorsque l'un est l'image de l'autre part multiplication à gauche (resp. à droite).

\begin{proof}
Si $P = P'$, il est clair que $aP$ et $a'P'$ ne s'intersectent pas, sauf si $a^{-1}a' \in P$, auquel cas $aP = a'P'$. De même, si $aPa^{-1} = a'P'a'^{-1}$, alors $aP$ et $a'P'$ ne s'intersectent pas, sauf si $a a'^{-1} \in aPa^{-1}$, ce qui implique aussi que $aP = a'P'$.

Réciproquement, supposons $P \neq P'$. Alors, par 2-transitivité de l'action adjointe, on peut supposer que $P' = w^{-1}Pw$. La décomposition de Bruhat donne alors
$$L = w^{-1}P \sqcup P'P.$$
Considérons l'élément $a'^{-1}a$. Si $a'^{-1}a = w^{-1}p$ pour un certain $p \in P$, on obtient $aPa^{-1} = a'P'a'^{-1}$. Sinon, il existe $p \in P$ et $p'\in P'$ tels que $a'^{-1}a = p' p^{-1}$. On a donc $ap = a'p'$ et $aP$ et $a'P'$ s'intersectent donc.
\end{proof}

De ces remarques, et d'après le lemme \ref{DisqueParabolique}, on peut déduire le résultat suivant :
\begin{proposition}
Il existe un unique sous-groupe parabolique $P$ tel que
\begin{itemize}
\item soit $Py \subset \partial U$ pour tout $y\in \partial U$,
\item soit $yP \subset \partial U$ pour tout $y\in \partial U$.
\end{itemize}
\end{proposition}

\begin{proof}

Montrons d'abord que pour tout $y \in \partial U$, la sous-algèbre parabolique $\p$ pour laquelle il existe $\eta >0$ telle que $D_\p(y,\eta) \subset U$ est unique.

Fixons $y\in \partial U$. Supposons qu'il existe $\eta >0$ et $\p$ sous-algèbre de Lie parabolique telle que $D_\p(y,\eta) \subset U$. Considérons d'autre part une suite $y_n$ convergeant vers $y$ et $\p_n$ convergeant vers $\p'$ telle que $D_{\p_n}(y_n, \epsilon) \subset U$ et $D_{\p'}(y, \epsilon)\subset U$.
Supposons que $\p' \neq \p$. Alors $\p$ et $\p'$ sont transverses, et $D_{\p}(y, \eta)$ et $D_{\p'}(y, \epsilon)$ le sont aussi. Par conséquent, pour $n$ assez grand; $D_{\p_n}(y_n, \epsilon)$ intersecte $D_{\p}(y, \eta)$, ce qui est absurde puisque $D_{\p}(y, \eta) \subset \partial U$.
Par conséquent, $\p = \p'$, et on en déduit l'unicité de $\p$.\\

Soit $y \in \partial U$ et $\p$ l'unique sous-algèbre de Lie parabolique telle que $D_{\p}(y,\epsilon) \subset \partial U$. Montrons que $\partial U$ contient en réalité tout $Py$. Pour cela, considérons l'ensemble $A$ des $z\in Py$ tels que $D_\p(z, \epsilon) \subset \partial U$, et montrons que cet ensemble est ouvert et fermé dans $Py$. Comme $Py$ est connexe et que $A$ contient $y$, cela impliquera le résultat.

Le fait que $A$ est fermé vient simplement du fait que $\partial U$ est fermé. Montrons que $A$ est ouvert. Soit $z \in A$ et $z' \in D_\p(z, \epsilon/2) \subset \partial U$. Alors il existe $\epsilon' > 0$ tel que $D_\p(z', \epsilon') \subset \partial U$. Mais d'autre part, il existe $\p'$ telle que $D_{\p'}(z', \epsilon) \subset \partial U$. Par unicité de $\p$, on a $\p' = \p$ et $z' \in A$. L'ensemble $A$ est donc ouvert et fermé, d'où $A=Py$ et $Py \subset \partial U$.

Résumons. Pour tout point $y \in \partial U$, il existe un unique sous-espace parabolique contenant $y$ et inclus dans $\partial U$. Soient deux sous-espaces paraboliques inclus dans le bord de $U$. S'ils ne sont pas parallèles, ils s'intersectent transversalement en un point $y\in \partial U$, ce qui contredit l'unicité du sous-espace parabolique passant par $y$. Donc tous les sous-espaces paraboliques contenus dans le bord de $U$ sont deux à deux parallèles.

Soient maintenant trois sous-espaces paraboliques distincts inclus dans $\partial U$. Si le premier et le deuxième sont parallèles à gauche, et le deuxième et le troisième sont parallèles à droite, alors le premier et le troisième ne sont pas parallèles, ce qui est absurde. On en déduit donc que les sous-espaces paraboliques contenus dans le bord de $U$ sont soit tous parallèles à gauche, soit tous parallèles à droite.
\end{proof}

\begin{proposition}
Quitte à intervertir l'action à gauche et l'action à droite, supposons que tous les sous-espaces paraboliques contenus dans le bord de $U$ sont les translations à gauche d'un même sous-groupe parabolique $P$. Alors :

\begin{itemize}
\item[(i)] $\Gamma \subset L \times P$.
\item[(ii)] $U$ est stable par multiplication à droite par $P$.
\end{itemize} 
\end{proposition}

\begin{proof}
Soit $y \in \partial U$ et $\gamma  = (\gamma_1, \gamma_2) \in \Gamma$. On a $yP \subset \partial U$. Comme le bord de $U$ est stable sous l'action de $\gamma$, 
$$\gamma .(yP) = (\gamma .y)  \gamma_2 P \gamma_2^{-1} \subset \partial U.$$
Or l'unique sous-espace parabolique inclus dans $\partial U$ et contenant $\gamma .y$ est $(\gamma .y) P$. Donc $\gamma_2 P \gamma_2^{-1} = P$ et $\gamma_2 \in P$, ce qui prouve (i).

Soit maintenant $x\in U$. Supposons que $xP \nsubseteq U$. Alors il existe $p \in P$ tel que $xp \in \partial U$. Mais alors $xpP \subset \partial U$ et en particulier $x \in \partial U$, ce qui est absurde. Donc $xP \subset U$.
\end{proof}

\subsubsection{Fin de la preuve}

Nous venons de prouver que le domaine $U$ fibre au dessus d'un ouvert $\Omega$ de $L/P$. De plus, l'action de $\Gamma$ sur $U$ préserve cette fibration, et l'action induite sur la base $\Omega$ est donnée par l'action à gauche de $\rho_1(\Gamma)$. Il nous reste à prouver que $\Omega = L/P$.

Prouvons pour commencer que $\Omega$ n'est pas $L/P$ privé d'un point. Supposons par l'absurde qu'il existe $x \in L$ tel que $U = L - xP$. Alors l'action de $\Gamma$ préserve $xP$, et on a donc $\Gamma \subset xP x^{-1}\times P$. D'après la décomposition de Bruhat, l'action de $xP x^{-1}\times P$ sur $U = L - xP$ est transitive.

La contradiction proviendra donc du résultat suivant :\\

\begin{Thm}
Soit $L$ un groupe de Lie de rang 1, $P$ un sous-groupe parabolique de $L$ et $x \in L$. Alors aucune variété compacte ne peut posséder une $(xP x^{-1}\times P, L - xP)$-structure. (Le facteur $xP x^{-1}$ agissant sur $L-xP$ par multiplication à gauche et le facteur $P$ par multiplication à droite.)
\end{Thm}

\begin{proof}
Quitte à conjuguer par un élément de $L \times \{\1\}$, on peut supposer que $x\notin P$. Par conséquent, $\1 \in L - xP$.
Posons $P' = xPx^{-1}$. Comme l'action adjointe de $L$ est $2$-transitive sur les sous-groupes paraboliques, il existe une décomposition de $\l = T_{\1} L$ sous la forme
$$\l = \m \oplus \R a \oplus \n_+ \oplus \n_-$$
où $a$ est un élément tel que $\ad_a$ est diagonalisable, $\m$ le centralisateur de $a$ dans $\k$, $\n_+$ la somme des espaces propres de $a$ pour les valeurs propres $>0$, $\n_-$ la somme des espaces propres de $a$ pour les valeurs propres $<0$, et telle que $\m \oplus \R a \oplus \n_+$ (resp. $\m \oplus \R a \oplus \n_-$) est l'algèbre de Lie tangente à $P'$ (resp. $P$).

Le stabilisateur de $\1$ dans $P'\times P$ est le plongement diagonal de $P \cap P'$. Son algèbre de Lie tangente est $\m \oplus \R a$, et le vecteur $a$ est invariant par l'action de ce stabilisateur. Par conséquent, $a$ s'étend en un champ de vecteur $X_a$ défini sur $L-xP$ et préservé par $P'\times P$. Ce champ de vecteur descendra donc sur toute variété munie d'une $(P'\times P, L-xP)$-structure.

D'autre part, $L-xP$ possède une forme volume $\omega$ invariante sous l'action de $P'\times P$ (restriction d'une forme volume bi-invariante sur $L$). Nous allons prouver que le champ de vecteur $X_a$ dilate le volume $\omega$.

Par homogénéité, il existe une constante $\alpha$ telle que $X_a . \omega = \alpha \omega$, et il suffit de calculer $\alpha$ au point $\1$. Considérons $e_1, \ldots, e_n$ une base de diagonalisation de $\l$ sous l'action de $\ad_a$. Soit $\lambda_i$ la valeur propre associée à $e_i$.
Soient $E_i$ les champs invariants à droite prolongeant les $e_i$.
Le volume $\omega$ étant invariant à droite, il s'écrit (à une constante près) $E_1^* \wedge \ldots \wedge E_n^*$. Comme il est également invariant à gauche, il est préservé par les champs $E_i$.

Posons $X_a = \sum \alpha_i E_i$, où les $\alpha_i$ sont des fonctions de $L-xP$ dans $\R$.
On a alors
$$X_a . \omega = \mathrm{div}(X_a) \omega$$
où 
$$\mathrm{div}(X_a) =\sum_i \d \alpha_i (E_i).$$

Calculons cette divergence au point $\1$. Soit d'abord $i$ tel que $\lambda_i \leq 0$. Alors $\exp(te_i) \subset P$.
On a donc $X_a(\exp(te_i)) = R_{\exp(te_i)} a$. Comme le champ $X_a$ est invariant à droite le long de $P$, les fonctions $\alpha_j$ sont invariantes le long de $P$, et en particulier
$$\d \alpha_i (e_i) =0$$ lorsque $\lambda_i \leq 0$.

Soit maintenant $i$ tel que $\lambda_i>0$. Alors $\exp(te_i) \subset P'$. Donc 
\begin{eqnarray*}
X_a(\exp(te_i)) & = & L_{\exp(te_i)} a \\
\ & = & R_{\exp(te_i)} \Ad_{\exp(te_i)} a \\
\ & = & R_{\exp(te_i)} (a + \lambda_i t e_i),
\end{eqnarray*}
d'où
$$\alpha_i(\exp(te_i)) = \lambda_i t.$$
En prenant la dérivée en $t = 0$, on obtient finalement :
$\d \alpha_i(e_i) = \lambda_i$.

En conclusion, $\mathrm{div}(X_a)$ est la somme des valeurs propres strictement positives de $a$. Cette divergence est donc non nulle.
Or, si $M$ est une variété compacte, $\omega$ une forme volume sur $M$ et $X$ un champ de vecteur, de flot $\Phi_t$, on a 
$$\int_M \Phi_t^* \omega = \int_M \omega$$
pour tout $t$.
En particulier, le champ $X$ ne peut pas dilater uniformément le volume $\omega$. Par conséquent, Il n'existe pas de variété compacte localement modelée sur $(P' \times P, L-xP)$.

\end{proof}

Nous avons donc prouvé que $\Omega$ ne peut pas être $L/P$ privé d'un seul point. Montrons maintenant que $\Omega$ ne peut pas non plus être $L/P$ privé de $2$ points ou plus.

Commençons par la remarque suivante :

\begin{proposition}
Soit $\Delta$ un sous-groupe de $L$ préservant un ouvert $\Omega$ de $L/P$ dont le complémentaire contient au moins deux points. Alors l'action de $\Delta$ sur $\Omega$ est propre.
\end{proposition}

Pour prouver cette proposition, rappelons que l'action de $L$ sur $L/P$ possède une dynamique ``Nord-Sud'', dans le sens suivant:

\begin{lemme}[\cite{Frances07}]\label{N-S}
Soit $\gamma_n$ une suite de $L$ sortant de tout compact. Alors il existe deux \emph{pôles} $x_-$ et $x_+ \in L/P$ (éventuellement égaux), une suite extraite $\gamma_{n_k}$, une suite de voisinages $U_k$ de $x_-$ et une suite de voisinages $V_k$ de $x_+$ tels que
\begin{itemize}
\item Pour tout $k$, $\gamma_{n_k}$ envoie le complémentaire de $U_k$ dans $V_k$ et $\gamma_{n_k}^{-1}$ envoie le complémentaire de $V_k$ dans $U_k$
\item $\bigcap_k U_k = \{x_-\}$ et $\bigcap_k V_k = \{x_+\}$.
\end{itemize}
\end{lemme}

Ce résultat est une propriété classique des espaces $L/P$ où $L$ est de rang $1$. Le point $x_-$ (resp. $x_+$) est appelé \emph{pôle répulseur} (resp. \emph{pôle attracteur}) de la suite $\gamma_{n_k}$.

Considérons maintenant un sous-groupe $\Delta$ de $L$ préservant un domaine $\Omega$ de $L/P$.

\begin{proposition} \label{PolesBord}
Si le complémentaire de $\Omega$ contient au moins deux points, alors toute suite $\gamma_n$ de $\Delta$ possédant une dynamique Nord-Sud a ses pôles dans $\Omega^c$.
\end{proposition}

Montrons tout d'abord que cette proposition implique que l'action de $\Delta$ sur $\Omega$ est propre. Supposons qu'elle ne le soit pas. Il existe alors un compact $K \subset \Omega$ et une suite $\gamma_n$ sortant de tout compact de $L$ telle que $\gamma_n . K$ intersecte $K$ pour tout $n$. Considérons $n_k$ une extraction, et $x_-$, $x_+$, $U_k$, $V_k$ comme dans le lemme \ref{N-S}.

D'après la proposition \ref{PolesBord}, $x_-$ et $x_+$ de sont pas dans $\Omega$. En particulier, $K$ ne contient aucun des deux pôles. Par conséquent, Pour $k$ assez grand, $K$ est inclus dans le complémentaire de $U_k$ et donc $\gamma_{n_k} \cdot K \subset V_k$. Mais pour $k$ assez grand, $K$ est aussi inclus dans le complémentaire de $V_k$. Donc pour $k$ assez grand, $\gamma_{n_k} \cdot K$ n'intersecte pas $K$. Cela contredit l'hypothèse. Donc l'action est propre.

\begin{proof}[Preuve de la proposition \ref{PolesBord}]

Soit $\gamma_n$ une suite de $\Delta$ possédant une dynamique Nord-Sud, $x_-$ et $x_+$ ses pôles. Alors :
\begin{itemize}
\item Soit $\Omega^c$ est inclus dans $\{x_-, x_+\}$; dans ce cas, comme $\Omega^c$ contient au moins deux points, $\Omega^c = \{x_-, x_+\}$
\item Soit il existe $x \in \Omega^c$ différent de $x_-$ et $x_+$. La dynamique Nord-Sud nous dit alors que $\gamma_n\cdot x$ converge vers $x_+$ et $\gamma_n^{-1} \cdot x$ converge vers $x_-$. Comme $\Omega^c$ est un fermé $\Delta$-invariant, $x_-$ et $x_+$ appartiennent à $\Omega^c$.
\end{itemize}
\end{proof}

Revenons à notre ouvert $U$ de $L$ qui fibre au dessus de $\Omega \subset L/P$, et sur lequel $\Gamma \subset L\times P$ agit proprement, discrètement et cocompactement. Nous avons prouvé que le complémentaire de $\Omega$ contient au moins deux points. Par conséquent, l'action de $\rho_1(\Gamma)$ sur $\Omega$ est propre. Nous allons montrer que $\Gamma$ ne peut pas être discret dans $L\times P$.

$P$ se décompose sous la forme $M \exp(\R a) N_+$, où $N_+$ est le sous-groupe dérivé de $P$, $a$ un élément de $\p$ dont l'action adjointe sur $\n_+$ est diagonalisable à valeurs propres strictement positives, et $M$ le centralisateur de $a$ dans $K$. Remarquons que le sous-groupe de $P$ engendré par $\exp(a)$ et par un élément $b \in N_+ \backslash \1$ n'est pas discret. En effet, l'action adjointe de $\exp(a)$ dilate $\n_+$, et par conséquent, $\exp(-na) b \exp(na)$ converge vers $\1$ sans être stationnaire. L'idée de la preuve est d'utiliser la cocompacité de l'action de $\Gamma$ sur $U$ pour trouver des éléments de $\Gamma$ ``assez proches'' de $(\1,\exp(na))$ et de $(\1,b)$, et conclure que $\Gamma$ non plus n'est pas discret.

Munissons $L$ d'une métrique invariante à droite.
Soit $x \in L$ tel que $xP$ est inclus dans $U$. En particulier, $x \exp(na) \in U$ pour tout $n$. Par cocompacité de l'action de $\Gamma$, il existe donc une suite $\gamma_n=(\alpha_n, \beta_n) \in \Gamma$ telle que $y_n = \alpha_n x \exp(na) \beta_n^{-1}$ reste dans un domaine fondamental compact de $U$. En projetant dans $L/P$, on observe que $\alpha_n x P$ reste dans un compact de $\Omega$. Comme l'action de $\rho_1(\Gamma)$ sur $\Omega$ est propre, on en déduit que la suite $\alpha_n$ est bornée dans $L$. Posons $c_n = y_n^{-1} \alpha_n x$. La suite $\beta_n = c_n \exp(na)$ reste donc à distance bornée de $\exp(na)$. 

Supposons qu'il existe un élément $\eta =(\zeta, \xi)$ du groupe dérivé de $\Gamma$, tel que $\xi$ soit différent de l'élément neutre. $\xi$ est alors inclus dans le groupe dérivé $N_+$, qui est bien sûr distingué dans $P$. Par conséquent, la suite $c_n^{-1} \xi c_n$ reste dans un domaine borné de $N_+$, et la suite ${\beta_n}^{-1} \xi \beta_n= \Ad_{\exp(-na)} (c_n^{-1} \xi c_n)$ converge vers $\1$ sans être stationnaire. D'autre part, la suite $\alpha_n$ est bornée, et la suite $\alpha_n^{-1} \zeta \alpha_n$ l'est donc aussi. On obtient finalement que la suite $\gamma_n^{-1} \eta \gamma_n$ est bornée dans $L\times L$ et prend une infinité de valeurs. $\Gamma$ n'est donc pas discret.

Il reste à voir qu'on peut trouver $(\zeta, \xi) \in [\Gamma, \Gamma]$ avec $\xi \neq \1$. Dans le cas contraire, on a $\rho_2(\Gamma)$ abélien. L'adhérence de Zariski de $\rho_2(\Gamma)$ dans $P$ est alors un sous-groupe abélien fermé $A$. Ce sous-groupe est soit inclus dans $N+$, soit dans un sous-groupe conjugué à $M \exp(\R a)$. Dans tous les cas, $P/A$ n'est pas compact. On peut donc considérer une suite $y_n$ dans $P$ dont la distance à $A$ tend vers l'infini. Par cocompacité de l'action de $\Gamma$, il existe $(\alpha'_n, \beta'_n) \in \Gamma$ tel que $\alpha'_n x y_n {\beta'}_n^{-1}$ reste dans un domaine compact de $U$. Comme précédemment on en déduit que $\alpha'_n$ est bornée, puis que $\beta'_n$ reste à distance bornée de $y_n$, ce qui est absurde car $\beta'_n \in A$ et que $y_n$ s'éloigne indéfiniment de $A$.

On obtient finalement que le groupe $\Gamma$ n'est pas discret dans $L\times P$, ce qui est absurde, puisqu'il agit discrètement sur $U$. Ceci achève de prouver que le complémentaire de $\Omega$ ne contient pas plus de deux points. Par conséquent, $\Omega = L/P$, $U=L$, et le théorème est prouvé.

\bibliographystyle{smfplain}
\bibliography{mabiblio}
\end{document}